\documentclass[a4paper,12pt,leqno]{article}
\usepackage[utf8]{inputenc}

\usepackage[top=3cm, left=2.5cm, right=2.5cm, bottom=3cm]{geometry}

\usepackage{hyperref}
\hypersetup{
    colorlinks=true,
    linkcolor=dodger,
    filecolor=dodger,
    urlcolor=dodger,
    citecolor=dodger,
}

\usepackage{amsmath,mathtools}
\usepackage{amsthm, pb-diagram}
\usepackage{amssymb,comment}
\usepackage{amsfonts,graphicx,color}
\usepackage{enumitem, fancyhdr, dsfont}
\usepackage[normalem]{ulem}
\usepackage{fontawesome}
\usepackage{thmtools,cleveref}
\usepackage{stmaryrd}

\usepackage{lineno,multicol}

\usepackage{tikz,float}

\setlist{topsep=0ex}

\makeatletter
\def\namedlabel#1#2{\begingroup
    #2%
    \def\@currentlabel{#2}%
    \phantomsection\label{#1}\endgroup
}
\makeatother

    \DeclareMathOperator{\dom}{{\rm dom}}

    \newcommand{\Bwf}{\mathcal{B}}
    \newcommand{\Cwf}{\mathcal{C}}
    
    \newcommand{\Ewf}{\mathcal{E}}

    \newcommand{\Iwf}{\mathcal{I}}
    \newcommand{\Jwf}{\mathcal{J}}
    
    \newcommand{\Mwf}{\mathcal{M}}
    \newcommand{\Nwf}{\mathcal{N}}
    \newcommand{\SNwf}{\mathcal{SN}}
    
    \newcommand{\NAwf}{\mathcal{N\!A}}
    
    \newcommand{\Pwf}{\mathcal{P}}
    \newcommand{\Scal}{\mathcal{S}}
    \newcommand{\Swf}{\mathcal{S}}

    \newcommand{\bfrak}{\mathfrak{b}}
    \newcommand{\cfrak}{\mathfrak{c}}
    \newcommand{\dfrak}{\mathfrak{d}}

    \newcommand{\Cbf}{\mathbf{C}}
    \newcommand{\Dbf}{\mathbf{D}}

    \newcommand{\Ior}{\mathbb{I}}

    \newcommand{\menos}{\smallsetminus}
    \newcommand{\vacio}{\varnothing}
    \DeclareMathOperator{\pts}{\mathcal{P}}

    \newcommand{\Q}{\mathbb{Q}}
    \newcommand{\R}{\mathbb{R}}

    \newcommand{\imp}{\mathrel{\mbox{$\Rightarrow$}}}

    \newcommand{\la}{\langle}
    \newcommand{\ra}{\rangle}
    \DeclareMathOperator{\cf}{\mbox{\rm cf}}
    
    \newcommand{\frestr}{{\upharpoonright}}

    \DeclareMathOperator{\Fn}{\mbox{\rm Fn}}
    
    \DeclareMathOperator{\add}{\mathrm{add}}
    \DeclareMathOperator{\non}{\mbox{\rm non}}
    \DeclareMathOperator{\cov}{\mbox{\rm cov}}
    \DeclareMathOperator{\cof}{\mbox{\rm cof}}
    
    \newcommand{\leqT}{\mathrel{\mbox{$\preceq_{\mathrm{T}}$}}}
    \newcommand{\eqT}{\mathrel{\mbox{$\cong_{\mathrm{T}}$}}}
    \DeclareMathOperator{\Seq}{\mathrm{seq}}
    \DeclareMathOperator{\Seqw}{\mathrm{seq}_{<\omega}}

    \newcommand{\Mg}{\mathbf{Mg}}
    \newcommand{\Ed}{\mathbf{Ed}}
    \newcommand{\Rbf}{\mathbf{R}}
    
    \newcommand{\seq}[2]{\left\la #1 :\, #2\right\ra}
    \newcommand{\set}[2]{\left\{#1 :\, #2\right\}}
    
    \newcommand{\vfa}{\mathfrak{v}^\forall}
    \newcommand{\cfa}{\mathfrak{c}^\forall}
    \newcommand{\vxt}{\mathfrak{v}^\exists}
    \newcommand{\cxt}{\mathfrak{c}^\exists}

    \newcommand{\Lc}{\mathbf{Lc}}
    \newcommand{\aLc}{\mathbf{aLc}}

    \newcommand{\blc}{\mathfrak{b}^{\mathrm{Lc}}}
    \newcommand{\dlc}{\mathfrak{d}^{\mathrm{Lc}}}
    \newcommand{\balc}{\mathfrak{b}^{\mathrm{aLc}}}
    \newcommand{\dalc}{\mathfrak{d}^{\mathrm{aLc}}}
    
    \newcommand{\cc}{\mathrm{cc}}
    \newcommand{\cp}{\mathrm{c}}
    \newcommand{\nc}{\mathrm{Nc}}
    \newcommand{\Lb}{\mathrm{Lb}}

    \newcommand{\id}{\mathrm{id}}

\newcommand{\minLc}{\mathrm{minLc}}
\newcommand{\supLc}{\mathrm{supLc}}
\newcommand{\supaLc}{\mathrm{supaLc}}
\newcommand{\minaLc}{\mathrm{minaLc}}    
\newcommand{\baire}{\omega^\omega}
\newcommand{\cantor}{2^\omega}

\newcommand{\tbf}{\mathbf{t}}
\newcommand{\dloc}{\mathfrak{d}^{\mathrm{wLc}}}
\newcommand{\bloc}{\mathfrak{b}^{\mathrm{wLc}}}
\newcommand{\tilb}{\tilde b}

    \definecolor{carrotorange}{rgb}{0.93, 0.57, 0.13}
    \definecolor{dodger}{rgb}{0.0,0.5,1.0}
    
    \newcommand{\Swfw}{\Swf^{\mathrm{w}}}
    \newcommand{\wLc}{\mathbf{wLc}}
    
\title{Localization and anti-localization cardinals}
\author{Miguel A.~Cardona%
\thanks{Email: \href{mailto:miguel.cardona@upjs.sk}{\texttt{miguel.cardona@upjs.sk}}\newline
Partially supported by the Slovak Research and Development Agency under Contract No.~APVV-20-0045 and by Pavol Jozef \v{S}af\'arik University at a postdoctoral position.\smallskip}
\and 
        Diego A.~Mejía%
        \thanks{Email: \href{mailto:diego.mejia@shizuoka.ac.jp}{\texttt{diego.mejia@shizuoka.ac.jp}}\newline 
Supported by the Grant-in-Aid for Early Career Scientists 18K13448, Japan Society for the Promotion of Science.}
}

\date{{\normalsize
${}^*$Institute of Mathematics, Pavol Jozef \v{S}af\'arik University,\\
041 80, Jesenn\'a 5, 040 01 Ko\v{s}ice, Slovakia\\\medskip
${}^\dagger$Creative Science Course (Mathematics), Faculty of Science, Shizuoka University\\ Ohya 836, Suruga-ku, Shizuoka 422--8529, Japan}
}

\begin{document}

\makeatletter
\def\@roman#1{\romannumeral #1}
\makeatother

\newcounter{enuAlph}
\renewcommand{\theenuAlph}{\Alph{enuAlph}}

\numberwithin{equation}{section}
\renewcommand{\theequation}{\thesection.\arabic{equation}}

\theoremstyle{plain}
  \newtheorem{theorem}[equation]{Theorem}
  \newtheorem{corollary}[equation]{Corollary}
  \newtheorem{lemma}[equation]{Lemma}
  \newtheorem{mainlemma}[equation]{Main Lemma}
  \newtheorem{prop}[equation]{Proposition}
  \newtheorem{clm}[equation]{Claim}
  \newtheorem{fct}[equation]{Fact}
  \newtheorem{question}[equation]{Question}
  \newtheorem{problem}[equation]{Problem}
  \newtheorem{conjecture}[equation]{Conjecture}
  \newtheorem*{thm}{Theorem}
  \newtheorem{teorema}[enuAlph]{Theorem}
  \newtheorem*{corolario}{Corollary}
  \newtheorem*{scnmsc}{(SCNMSC)}
\theoremstyle{definition}
  \newtheorem{definition}[equation]{Definition}
  \newtheorem{example}[equation]{Example}
  \newtheorem{remark}[equation]{Remark}
  \newtheorem{notation}[equation]{Notation}
  \newtheorem{context}[equation]{Context}
  \newtheorem{exer}[equation]{Exercise}
  \newtheorem{exerstar}[equation]{Exercise*}

  \newtheorem*{defi}{Definition}
  \newtheorem*{acknowledgements}{Acknowledgements}
  
\def\sectionautorefname{Section}
\def\subsectionautorefname{Subsection}

\maketitle

\begin{abstract}
This paper is intended to survey the basics of localization and anti-localization cardinals on the reals, and its interplay with notions and cardinal characteristics related to measure and category.
%
\end{abstract}









\section{Introduction}



The localization and anti-localization cardinals were introduced in the remarkable work of Miller~\cite{Mi1982} and Bartoszyi\'nski~\cite{Ba1984,Ba1987} to characterize the cardinal characteristics associated with measure and category, which also give tools to prove the inequalities in Cicho\'n's diagram (\autoref{FigCichon}).


\begin{figure}[ht!]
\centering
\begin{tikzpicture}
\small{
\node (aleph1) at (-1,3) {$\aleph_1$};
\node (addn) at (1,3){$\add(\Nwf)$};
\node (covn) at (1,7){$\cov(\Nwf)$};
\node (nonn) at (9,3) {$\non(\Nwf)$} ;
\node (cfn) at (9,7) {$\cof(\Nwf)$} ;
\node (addm) at (3.66,3) {$\add(\Mwf)$} ;
\node (covm) at (6.33,3) {$\cov(\Mwf)$} ;
\node (nonm) at (3.66,7) {$\non(\Mwf)$} ;
\node (cfm) at (6.33,7) {$\cof(\Mwf)$} ;
\node (b) at (3.66,5) {$\bfrak$};
\node (d) at (6.33,5) {$\dfrak$};
\node (c) at (11,7) {$\cfrak$};
\draw (aleph1) edge[->] (addn)
      (addn) edge[->] (covn)
      (covn) edge [->] (nonm)
      (nonm)edge [->] (cfm)
      (cfm)edge [->] (cfn)
      (cfn) edge[->] (c);
\draw
   (addn) edge [->]  (addm)
   (addm) edge [->]  (covm)
   (covm) edge [->]  (nonn)
   (nonn) edge [->]  (cfn);
\draw (addm) edge [->] (b)
      (b)  edge [->] (nonm);
\draw (covm) edge [->] (d)
      (d)  edge[->] (cfm);
\draw (b) edge [->] (d);

}
\end{tikzpicture}
\caption{Cicho\'n's diagram. The arrows mean $\leq$. Also
  $\add(\Mwf)=\min\{\bfrak,\cov(\Mwf)\}$ and $\cof(\Mwf)=\max\{\dfrak,\non(\Mwf)\}$.}
  \label{FigCichon}
\end{figure}

Over the years, the localization and anti-localization cardinals have been relevant objects of study propelled  by a number of people~\cite{paw85, GS93,K8,KS09,KS12, BrM,KO14,CM,KM, GK,CKM}. Although a large number of papers have been published,  we feel the need to have a survey about localization and anti-localization cardinals that could serve as a reference note. 

In this paper, we attempt to present a careful and detailed study of the localization and anti-localization cardinals, filling in any gaps that we perceived. The purpose is to establish the well-known characterizations of classical cardinal characteristics related to measure and category that involve the localization and anti-localization cardinals, with the hope to making them more accessible to people interested in combinatorics of the real line.  

In \autoref{sec:relchac} and~\ref{sec:Lc}, we review the basic notation and results about relational systems and localization and anti-localization cardinals. The characterizations of cardinal characteristics related to measure and category appear in the sections that follow:

\autoref{sec:NLc}: The additivity and cofinality of the null ideal (Bartoszy\'nski~\cite{Ba1984}).

\autoref{MaLc}: The covering and uniformity of the meager ideal (Miller~\cite{Mi1982}, Bartoszy\'nski \cite{Ba1987}).

\autoref{NaLc}: The covering and uniformity of the null ideal (Bartoszy\'nski~\cite{B1988}).

\autoref{sec:ELc}: The covering and uniformity of $\Ewf$, the ideal generated by the $F_\sigma$ measure zero sets (Bartoszy\'nski and Shelah~\cite{BS1992}, \cite{Car23}).

\autoref{sec:other}: The uniformity of the ideal of null-additive sets (Pawlikowski~\cite{paw85}).




The previous characterizations also appear in~\cite{BJ}.

\section{Relational systems and cardinal characteristics}\label{sec:relchac}

We start by presenting the framework of relational systems, which is very useful to define many classical cardinal characteristics. This was introduced in~\cite{Vojtas}, see also~\cite{blass}.

\begin{definition}\label{def:relsys}
We say that $\Rbf=\la X, Y, \sqsubset\ra$ is a \textit{relational system} if it consists of two non-empty sets $X$ and $Y$ and a relation $\sqsubset$.
\begin{enumerate}[label=(\arabic*)]
    \item A set $F\subseteq X$ is \emph{$\Rbf$-bounded} if $\exists\, y\in Y\ \forall\, x\in F\colon x \sqsubset y$. 
    \item A set $D\subseteq Y$ is \emph{$\Rbf$-dominating} if $\forall\, x\in X\ \exists\, y\in D\colon x \sqsubset y$. 
\end{enumerate}

We associate two cardinal characteristics with this relational system $\Rbf$: 
\begin{itemize}
    \item[{}] $\bfrak(\Rbf):=\min\{|F|:\, F\subseteq X  \text{ is }\Rbf\text{-unbounded}\}$ the \emph{unbounding number of $\Rbf$}, and
    
    \item[{}] $\dfrak(\Rbf):=\min\{|D|:\, D\subseteq Y \text{ is } \Rbf\text{-dominating}\}$ the \emph{dominating number of $\Rbf$}.
\end{itemize}
\end{definition}
The dual of $\Rbf$ is defined by $\Rbf^\perp:=\la Y, X, \sqsubset^\perp\ra$ where $y\sqsubset^\perp x$ iff $x\not\sqsubset y$. Note that $\bfrak(\Rbf^\perp)=\dfrak(\Rbf)$ and $\dfrak(\Rbf^\perp)=\bfrak(\Rbf)$. 

Below, we define directed preorders, which give a very representative general example of relational systems. 

\begin{definition}\label{examSdir}
We say that $\la S,\leq_S\ra$ is a \emph{directed preorder} if it is a preorder (i.e.\ $\leq_S$ is a reflexive and transitive relation on $S$) such that 
\[\forall\, x, y\in S\ \exists\, z\in S\colon x\leq_S z\text{ and }y\leq_S z.\] 
A directed preorder $\la S,\leq_S\ra$ is seen as the relational system $S=\la S, S,\leq_S\ra$, and their associated cardinal characteristics are denoted by $\bfrak(S)$ and $\dfrak(S)$. The cardinal $\dfrak(S)$ is actually the \emph{cofinality of $S$}, typically denoted by $\cof(S)$ or $\cf(S)$.
\end{definition}

Recall that, whenever $S$ is a a directed preorder without maximum element, $\bfrak(S)$ is infinite and regular, and $\bfrak(S)\leq\cf(\dfrak(S))\leq\dfrak(S)\leq|S|$. Even more, if $L$ is a linear order without maximum then $\bfrak(L)=\dfrak(L)=\cf(L)$.

\begin{example}\label{exm:Baire}
Consider $\omega^\omega=\la\omega^\omega,\leq^*\ra$, 
which is a directed preorder. The cardinal characteristics $\bfrak:=\bfrak(\omega^\omega)$ and $\dfrak:=\dfrak(\omega^\omega)$ are the well-known \emph{bounding number} and \emph{dominating number}, respectively.
\end{example}

The cardinal characteristics associated with an ideal can be characterized by relational systems as well.

\begin{example}\label{exm:Iwf}
For $\Iwf\subseteq\pts(X)$, define the relational systems: 
\begin{enumerate}[label=(\arabic*)]
    \item $\Iwf:=\la\Iwf,\Iwf,\subseteq\ra$, which is a directed partial order when $\Iwf$ is closed under unions (e.g.\ an ideal).
    
    \item $\Cbf_\Iwf:=\la X,\Iwf,\in\ra$.
\end{enumerate}
\end{example}

\begin{fct}\label{CarCh}
If $\Iwf$ is an ideal on $X$, containing $[X]^{<\aleph_0}$, then
\begin{multicols}{2}
\begin{enumerate}[label= \rm (\alph*)]
    \item $\bfrak(\Iwf)=\add(\Iwf)$, the additivity of $\Iwf$. 
    
    \item $\dfrak(\Iwf)=\cof(\Iwf)$, the cofinality of $\Iwf$.  
    
    \item $\dfrak(\Cbf_\Iwf)=\cov(\Iwf)$, the covering of $\Iwf$. 
    
    \item $\bfrak(\Cbf_\Iwf)=\non(\Iwf)$, the uniformity of $\Iwf$.
\end{enumerate}
\end{multicols}
\end{fct}


The Tukey connection is a useful notion to determine relations between cardinal characteristics.

\begin{definition}\label{def:Tukey}
Let $\Rbf=\la X,Y,\sqsubset\ra$ and $\Rbf'=\la X',Y',\sqsubset'\ra$ be relational systems. We say that $(\Psi_-,\Psi_+)\colon\Rbf\to\Rbf'$ is a \emph{Tukey connection from $\Rbf$ into $\Rbf'$} if 
 $\Psi_-\colon X\to X'$ and $\Psi_+\colon Y'\to Y$ are functions such that  \[\forall\, x\in X\ \forall\, y'\in Y'\colon \Psi_-(x) \sqsubset' y' \Rightarrow x \sqsubset \Psi_+(y').\]
The \emph{Tukey order} between relational systems is defined by
$\Rbf\leqT\Rbf'$ iff there is a Tukey connection from $\Rbf$ into $\Rbf'$. \emph{Tukey equivalence} is defined by $\Rbf\eqT\Rbf'$ iff $\Rbf\leqT\Rbf'$ and $\Rbf'\leqT\Rbf$
\end{definition}

\begin{fct}\label{fct:Tukey}
Assume that $\Rbf=\la X,Y,\sqsubset\ra$ and $\Rbf'=\la X',Y',\sqsubset'\ra$ are relational systems and that $(\Psi_-,\Psi_+)\colon \Rbf\to\Rbf'$ is a Tukey connection.
\begin{enumerate}[label=\rm(\alph*)]
    \item If $D'\subseteq Y'$ is $\Rbf'$-dominating, then $\Psi_+[D']$ is $\Rbf$-dominating.
    \item $(\Psi_+,\Psi_-)\colon (\Rbf')^\perp\to\Rbf^\perp$ is a Tukey connection.
    \item If $E\subseteq X$ is $\Rbf$-unbounded then $\Psi_-[E]$ is $\Rbf'$-unbounded.
\end{enumerate}
\end{fct}

\begin{corollary}\label{cor:Tukeyval}
\ 
\begin{enumerate}[label=\rm(\alph*)]
    \item $\Rbf\leqT\Rbf'$ implies $(\Rbf')^\perp\leqT\Rbf^\perp$.
    \item $\Rbf\leqT\Rbf'$ implies $\bfrak(\Rbf')\leq\bfrak(\Rbf)$ and $\dfrak(\Rbf)\leq\dfrak(\Rbf')$.
    \item $\Rbf\eqT\Rbf'$ implies $\bfrak(\Rbf')=\bfrak(\Rbf)$ and $\dfrak(\Rbf)=\dfrak(\Rbf')$.
\end{enumerate}
\end{corollary}


\begin{example}\label{Charb-d}
Denote by $\Ior$ the set of partitions of $\omega$ into finite non-empty intervals. Define the following relations on $\Ior$:
   \[ I \sqsubseteq J \text{ iff } \forall^\infty\, n<\omega\ \exists\, m<\omega\colon I_m\subseteq J_n; \quad
      I \triangleright J \text{ iff } \exists^\infty\, n<\omega\ \exists m<\omega\colon  I_n\supseteq J_m.
   \]
Note that $\sqsubseteq$ is a directed preorder on $\Ior$, so we think of $\Ior$ as the relational system with the relation $\sqsubseteq$. Denote $\Dbf_2:=\la\Ior,\Ior,\ntriangleright\ra$.
   
   For each $I\in\Ior$ we define $f_I\colon \omega\to\omega$ and $I^{*2}\in\Ior$ such that $f_I(n):=\min I_{n}$ and $I^{*2}_n:=I_{2n}\cup I_{2n+1}$. For each increasing $f\in\omega^\omega$ define the increasing function $f^*\colon \omega\to\omega$ such that $f^*(0)=0$ and $f^*(n+1)=f(f^*(n)+1)$, and define $I^f\in\Ior$ such that $I_n^f:=[f^*(n),f^*(n+1))$.
   
   In Blass~\cite{blass} it is proved that $\Ior \eqT\Dbf_2\eqT \omega^\omega$. 
   More concretely, for any increasing function $g\in\omega^\omega$ and $I,J\in\Ior$: 
   \begin{enumerate}[label=\rm(\alph*)]
       \item $f_I\leq^* g$ implies $I\sqsubseteq I^g$ (so $\Ior\leqT \omega^\omega$). 
       
       \item $g\not\leq^* f_I$ implies $I^g \triangleright  I$ (so $\omega^\omega\leqT \Dbf_2$).

       \item $I^{*2}\sqsubseteq J$ implies $I\ntriangleright J$ (so $\Dbf_2\leqT \Ior$).
   \end{enumerate}
\end{example}

To finish this section, we review the composition of relational systems.

\begin{definition}\label{compTK}
The \emph{composition of two relational systems $\Rbf_e=\la X_e,Y_e,\sqsubset_e\ra$}, for $e\in\{0,1\}$, is defined by $(\Rbf_0;\Rbf_1):=\la X_0\times X_1^{Y_0}, Y_0\times Y_1, \sqsubset_* \ra$ where
\[(x,f)\sqsubset_*(y,b)\text{ iff }x \sqsubset_0 y\text{ and }f(y) \sqsubset_1 b.\]
\end{definition}

\begin{fct}\label{ex:compleqT}
Let $\Rbf_i$ be a relational system for $i<3$. If $\Rbf_0\leqT\Rbf_1$, then $(\Rbf_0;\Rbf_2)\leqT(\Rbf_1;\Rbf_2)$ and $(\Rbf_2;\Rbf_0)\leqT(\Rbf_2;\Rbf_1)$.
\end{fct}

The following theorem describes the effect of the composition on cardinal characteristics.

\begin{theorem}[{\cite[Thm.~4.10]{blass}}]\label{blascomp}
Let $\Rbf_e$ be a relational system for $e\in\{0,1\}$. Then $\dfrak(\Rbf_0;\Rbf_1)=\dfrak(\Rbf_0)\cdot\dfrak(\Rbf_1)$ and $\bfrak(\Rbf_0;\Rbf_1)=\min\{\bfrak(\Rbf_0),\bfrak(\Rbf_1)\}$. 
\end{theorem}

\section{Localization and anti-localization}\label{sec:Lc}

The contents of this section are based on~\cite[Sec.~3.1]{CM}.
The localization and anti-localization cardinals are also defined via relational systems.

\begin{definition}\label{def:LcaLc}
We fix a sequence $b = \seq{b(n)}{n\in \omega}$ of non-empty sets and $h\colon \omega\to\omega$.
    \begin{enumerate}[label=(\arabic*)]
        \item Denote $\prod b := \prod_{n\in \omega}b(n)$ and $\Scal(b,h) := \prod_{n\in \omega} [b(n)]^{\leq h(n)}$. The functions in the second set are usually called \emph{slaloms}.

        \item When $x$ and $\varphi$ are functions with domain $\omega$, we write
        \begin{enumerate}[label=(\roman*)]
            \item $x\in^*\varphi$ iff the set $\set{n\in \omega}{x(n)\notin \varphi(n)}$ is finite, which is read as ``\emph{$\varphi$ localizes $x$}'';
            \item $x\in^\infty\varphi$ iff the set $\set{n\in \omega}{x(n)\in \varphi(n)}$ is finite. The expression $x\notin^\infty \varphi$ is read as ``\emph{$\varphi$ anti-localizes $x$}''.
        \end{enumerate}
        
        \item Let $\Lc(b,h) := \la\prod b, \Scal(b,h),\in^*\ra$, which is a relational system. Denote $\blc_{b,h}=\bfrak(\Lc(b,h))$ and $\dlc_{b,h}=\dfrak(\Lc(b,h))$, which we call \emph{localization cardinals}.\footnote{In~\cite{KM,CKM}, these are denoted by $\vfa_{b,h}$ and $\cfa_{b,h}$, respectively.}
        
        \item Let $\aLc(b,h) := \la\Scal(b,h),\prod b,\not\ni^\infty\ra$, which is a relational system. Denote $\balc_{b,h}=\bfrak(\aLc(b,h))$ and $\dalc_{b,h}=\dfrak(\aLc(b,h))$, which we call \emph{anti-localization cardinals}.\footnote{In~\cite{KM,CKM}, these are denoted by $\cxt_{b,h}$ and $\vxt_{b,h}$, respectively.}
    \end{enumerate}

    When a single set $a$ (like an ordinal) is used in the place of $b$ and $h$, we are referring to the constant sequence with value $a$. For example, when writing $\Lc(\omega,h)$ and $\aLc(b,1)$, $\omega$ refers to the constant sequence with value $\omega$, and $1$ to the constant sequence with value $1$. For any set $A$, $\id_{A}$ denotes the \textit{identity function on $A$}.

    It is also practical to allow that the domain of $b$ and $h$ is an arbitrary infinite countable set $D$ instead of just $\omega$, even if studying $D=\omega$ is enough to understand localization and antilocalization. For example, when $D\subseteq \omega$ is infinite, we can talk about $\Lc(b\frestr D,h\frestr D)$ and $\aLc(b\frestr D,h\frestr D)$, and their associated cardinal characteristics.
\end{definition}

The following results show natural Tukey connections between localization and anti-localization, and that for $b$
only the cardinality of $b(n)$ is relevant. Its proof is straightforward and therefore omitted.  

\begin{fct}\label{ex:easylc}
$\Lc(b,h)\leqT\Lc(b',h')$ and $\aLc(b',h')\leqT\aLc(b,h)$ whenever $|b(n)|\leq |b'(n)|$ and $h'(n)\leq h(n)$ for all but finitely many $n$. In particular, we have 
    \begin{enumerate}[label=\normalfont(\alph*)]
        \item $\blc_{b',h'}\leq\blc_{b,h}$ and $\dlc_{b,h}\leq\dlc_{b',h'}$,
        \item $\balc_{b,h}\leq\balc_{b',h'}$ and $\dalc_{b',h'}\leq \dalc_{b,h}$.
    \end{enumerate}
\end{fct}

\begin{fct}\label{trv:LcaLc}
    $\aLc(b,h)^\perp\leqT\Lc(b,h)$, so $\balc_{b,h}\leq\dlc_{b,h}$ and $\blc_{b,h}\leq\dalc_{b,h}$.
\end{fct}

In combinatorics of the reals, the study of localization and anti-localization cardinals are interesting when $0<h(n)<|b(n)|\leq \aleph_0$ for all $n<\omega$. Trivial cases are summarized as below.

\begin{enumerate}[label=\normalfont(T\arabic*)]
    \item\label{lctriv} If $|b(n)|\leq h(n)$ for all but finitely many $n<\omega$, $\dlc_{b,h}=1$ and $\blc_{b,h}$ is undefined. On the other hand, if $\exists^\infty\, n<\omega\colon h(n)=0$ then $\blc_{b,h}=1$ and $\dlc_{b,h}$ is undefined. Therefore, both localization cardinals for $b,h$ are defined only when $1\leq^*h$ and $\exists^\infty{n<\omega}:h(n)<|b(n)|$. 
    Even more, if $A:=\set{n<\omega}{1\leq h(n)<|b(n)|}$ (which is infinite) 
    then $\Lc(b,h)\eqT\Lc(b\frestr A,h\frestr A)$, so $\blc_{b,h}=\blc_{b\frestr A,h\frestr A}$ and $\dlc_{b,h}=\dlc_{b\frestr A,h\frestr A}$.
    
    \item If $\exists^\infty\, i<\omega\colon |b(n)|\leq h(n)$ then $\balc_{b,h}=1$ and $\dalc_{b,h}$ is undefined. On the other hand, if $\forall^\infty\, n<\omega\colon h(n)=0$ then $\dalc_{b,h}=1$ and $\balc_{b,h}$ is undefined. Hence, both anti-localization cardinals for $b,h$ are defined iff $h(n)<|b(n)|$ for all but finitely many $n<\omega$, and $\exists^\infty\, n<\omega\colon h(n)\geq 1$, even more, if $A$ is as in~\ref{lctriv} then $\aLc(b,h)\eqT\aLc(b\frestr A,h\frestr A)$, so $\balc_{b,h}=\balc_{b\frestr A,h\frestr A}$ and $\dalc_{b,h}=\dalc_{b\frestr A,h\frestr A}$.
\end{enumerate}

Restricting to infinite sets yields the following general fact.

\begin{fct}\label{compinj}
Let $b$ and $h$ as in \autoref{def:LcaLc}.
If 
$A\subseteq \omega$ is infinite, then $\Lc(b\frestr A,h\frestr A)\leqT\Lc(b,h)$ and $\aLc(b\frestr A,h\frestr A)\leqT\aLc(b,h)$. Hence $\blc_{b,h}\leq\blc_{b\frestr A,h\frestr A}$ and $\dlc_{b\frestr A,h\frestr A}\leq\dlc_{b,h}$, likewise for the anti-localization cardinals.

Even more, if $\omega\smallsetminus A$ is finite then $\Lc(b\frestr A,h\frestr A)\eqT\Lc(b,h)$ and $\aLc(b\frestr A,h\frestr A)\eqT\aLc(b,h)$.
\end{fct}

The case when some (or all) $b(n)$ are uncountable is somewhat interesting. We do not develop this case, but cite the result below.

\begin{theorem}[{\cite[Thm.~3.4]{CM}}]\label{CM:unct}
    \
    \begin{enumerate}[label = \normalfont (\alph*)]
     \item If $b(i)$ is infinite for all but finitely many $i<\omega$ then
      \[\balc_{b,h}=\max\{\cof([\kappa]^{\aleph_0}),\balc_{\omega,h}\}\text{\ and }\dalc_{b,h}=\min\{\add([\kappa]^{\aleph_0}),\dalc_{\omega,h}\}\]
      where $\kappa=\liminf_{i<\omega} |b(i)|$. In particular, if $\kappa=\omega$ then $\balc_{b,h}=\balc_{\omega,h}$ and $\dalc_{b,h}=\dalc_{\omega,h}$; otherwise, if $\kappa$ is uncountable then $\dalc_{b,h}=\aleph_1$.
     \item If $b(i)$ is infinite for infinitely many $i<\omega$ then
      \[\blc_{b,h}=\min\{\add([\lambda]^{\aleph_0}),\blc_{\omega,h}\}\text{\ and }\dlc_{b,h}=\max\{\cof([\lambda]^{\aleph_0}),\dlc_{\omega,h}\}\]
      where $\lambda=\limsup_{i<\omega} |b(i)|$. In particular, if $\lambda=\omega$ then $\blc_{b,h}=\blc_{\omega,h}$ and $\dlc_{b,h}=\dlc_{\omega,h}$; if $\lambda$ is uncountable and $h$ goes to infinity then $\blc_{b,h}=\aleph_1$.
    \end{enumerate}
\end{theorem}

In addition, if $F:=\set{i<\omega}{|b(i)|<\aleph_0}$ is infinite, then $\balc_{b,h}=\balc_{b\frestr F,h\frestr F}$ and $\dalc_{b,h}=\dalc_{b\frestr F,h\frestr F}$, see~\cite[Lem.~3.6]{CM}.

From now on, when dealing with localization and anti-localization, we assume that $0<h(n)<|b(n)|\leq\aleph_0$ for all (but finitely many) $n<\omega$. According to \autoref{CM:unct} (and the comment that follows), the study of localization and anti-localization cardinals reduces to the cases $b\in\omega^\omega$ and $b=\omega$ (but we are not going to make this distinction).

Both localization and anti-localization relational systems have connections with the following relational system of \emph{eventually different functions}.

\begin{definition}\label{def:Ed}
  Let $b:=\la b(n):\, n<\omega\ra$ be a sequence of non-empty sets. Define the relational system $\Ed_b:=\la\prod b,\prod b,\neq^\infty\ra$ where $x=^\infty y$ means $x(n)=y(n)$ for infinitely many $n$. The relation $x\neq^\infty y$ expresses that \emph{$x$ and $y$ are eventually different}. We just write $\Ed:=\Ed_\omega$ (when $b$ is the constant function $\omega$).
\end{definition}

\begin{fct}\label{aLc1}
    $\aLc(b,1)\eqT \Ed_b$, so $\bfrak(\Ed_b)=\balc_{b,1}$ and $\dfrak(\Ed_b)=\dalc_{b,1}$.

    Even more, if $b\in\omega^\omega$ then $\Lc(b,b-1)\eqT \aLc(b,1)$, so $\blc_{b,b-1}=\balc_{b,1}$ and $\dlc_{b,b-1}=\dalc_{b,1}$.
\end{fct}

For $b\in\omega^\omega$, we have the following connection.

\begin{lemma}[{\cite{KM}}]\label{Ed-aLc}
Given $b, h\in\omega^\omega$such that $b\geq1$ and $h\geq1$, define $d(n):=\big\lceil\frac{b(n)}{h(n)}\big\rceil$. Then $\Ed_d\leqT\aLc(b,h)$, so $\balc_{b,h}\leq\balc_{d,1}$ and $\dalc_{d,1}\leq \dalc_{b,h}$.
\end{lemma}
\begin{proof}
For each $n\in\omega$, we partition $b(n)$ into $d(n)$ many sets $\seq{b_{n,m}}{m<d(n)}$ of size ${\leq}h(n)$. Set $\Psi_-\colon \prod d\to\Swf(b,h)$ by $\Psi_-(x)(n):=b_{n,x(n)}$  and set $\Psi_+\colon \prod b\to\prod d$ such that, for any $y\in\prod b$, $\Psi_+(y)(n)$ is the unique $m\in d(n)$ such that $y(n)\in b_{n,m}$. It is easy to see that $(\Psi_-,\Psi_+)$ is the required Tukey connection.
\end{proof}

We discuss some of the ``uninteresting" cases of localization and anti-localization cardinals, which are based on the following results from Goldstern and Shelah~\cite{GS93}.

\begin{lemma}[{\cite[Lem.~1.8]{GS93}}]\label{LcTtrick}
    Let $b$ and $b'$ be countable sequences of non-empty sets, and $h,h'\in\omega^\omega$. Let $I\in \Ior$ and assume that, for $n<\omega$, we have a function $t_n\colon b'(n)\to \prod_{i\in I_n} b(i)$ such that
    \[\text{for any $s\in\prod_{i\in I_n}[b(i)]^{\leq h(i)}$, $\left|t_n^{-1}\left[\prod_{i\in I_n}s(i)\right]\right|\leq h'(n)$.}\]
    Then $\Lc(b',h')\leqT \Lc(b,h)$.
\end{lemma}
\begin{proof}
    Define $f\colon\prod b'\to \prod b$ such that, for $x\in\prod b'$, $f(x)$ is the concatenation of $t_n(x(n))$ for $n<\omega$; and define $g\colon \Scal(b,h)\to \Scal(b',h')$ such that, for $\varphi\in\Scal(b,h)$, $g(\varphi)_n:= t_n^{-1}\left[\prod_{i\in I_n}\varphi(i)\right]$ (which has size ${\leq}h'(n)$ by hypothesis). Thus, $(f,g)$ is the required Tukey connection.
\end{proof}

\begin{lemma}[{\cite[Lem.~1.10]{GS93}}]\label{Lced}
    Let $p(n):=2^n$ for $n<\omega$, and $0<N<\omega$. Then $\dlc_{p,N}=\cfrak$.
\end{lemma}
\begin{proof}
    The inequality $\dlc_{p,N}\leq\cfrak$ is clear. To see the converse, interpret $p(n)$ as the set of binary sequences of length $n$. For $z\in 2^\omega$ define $x_z\in\prod p$ by $x_z(i):= z\frestr i$. Towards a contradiction, assume that $S\subseteq\Scal(p,N)$ is a set of size ${<\cfrak}$ such that every real in $\prod p$ is localized by some member of $S$. Then, there is a single $\varphi\in S$ that localizes infinitely many reals of the form $x_z$, but this is impossible because the $x_z$'s are pairwise eventually different.
\end{proof}

\begin{lemma}[{\cite[Lem.~1.11]{GS93}}]\label{lem:GS93}
    Assume $2\leq N<\omega$. Then $\dalc_{N,1}=\dlc_{N,N-1}=\cfrak$.
\end{lemma}
\begin{proof}
    By \autoref{aLc1} it is enough to show that $\dlc_{N,N-1}=\cfrak$. The inequality $\leq$ is clear. For the converse, in virtue of \autoref{Lced} it is enough to show that $\dlc_{p,N-1}\leq \dlc_{N,N-1}$, even better, we show that $\Lc(p,N-1)\leqT \Lc(N,N-1)$. Let $I=\la I_n:\, n<\omega\ra$ be the interval partition of $\omega$ such that $|I_n|=N^{p(n)}$, let $\set{t_{n,i}}{i\in I_n}$ enumerate all functions from $p(n)$ into $N$, and define $t_n\colon p(n)\to N^{I_n}$ by $t_n(j):=\la t_{n,i}(j):\, i\in I_n\ra$. 
    
    By \autoref{LcTtrick}, it suffices to show that $\left|t_n^{-1}\left[\prod_{i\in I_n}s(i)\right]\right|\leq N-1$ for any $s\in ([N]^{\leq N-1})^{I_n}$. Towards a contradiction, assume we have pairwise different $j_k\in p(n)$ for $k<N$ such that $t_n(j_k)\in\prod_{i\in I_n}s(i)$. Choose some $t\colon p(n)\to N$ such that $t(j_k):=k$ for all $k<N$, so $t=t_{n,i}$ for some $i\in I_n$. Then, $k=t_{n,i}(j_k)\in s(i)$ for all $k<N$, i.e., $s(i)=N$, which contradicts that $|s(i)|\leq N-1$.
\end{proof}

As a consequence, localization cardinals are easy to calculate when $h$ does not diverge to infinity.

\begin{theorem}\label{Lochnotinfty}
   If $h$ does not diverge to infinity then $\dlc_{b,h}\geq \cfrak$ and $\blc_{b,h}=N+1$ where $N$ is the minimum natural number such that $A_N:=\{n<\omega:\, h(n)=N\}$ is infinite. Moreover, $\dlc_{b,h}=\cfrak$ when $\forall^\infty\, n<\omega\colon |b(n)|\leq\cfrak$.
\end{theorem}

In the previous theorem, we have $\dlc_{b,h}=\cof([\lambda]^{\aleph_0})$ when $\lambda:=\limsup_{n<\omega} |b(n)|>\cfrak$ by \autoref{CM:unct}.

\begin{proof}
   By \autoref{ex:easylc},~\ref{compinj} and the general assumption that $b(n)>h(n)>0$ for all $n<\omega$, we get $\Lc(N+1,N)\leqT \Lc(b\frestr A_N,h\frestr A_N)\leqT\Lc(b,h)$, so $\cfrak\leq \dlc_{b,h}$ by \autoref{lem:GS93}, and $\blc_{b,h}\leq\blc_{N+1,N}$. Note that no single slalom in $\Swf(N+1,N)$ localizes all constant functions in $(N+1)^\omega$, so $\blc_{b,h}\leq \blc_{N+1,N}\leq N+1$. Conversely,
   by the definition of $N$, $N\leq^* h$, so $N<\blc_{b,h}$. Hence $\blc_{b,h}=\blc_{N+1,N}=N+1$.

   In addition, if $\forall\, n<\omega\colon |b(n)|\leq\cfrak$ then $\dlc_{b,h}\leq|\Scal(b,h)|=\big|\prod b\big|=\cfrak$. 
\end{proof}

We have similar result for the anti-localization cardinals when the sequence $\la \frac{h(n)}{|b(n)|}:\, n<\omega\ra$ does not converge to $0$ (put $\frac{h(n)}{|b(n)|}=0$ when $b(n)$ is an infinite set). 

\begin{theorem}\label{aLctrivial}
  If the sequence $\seq{\frac{h(n)}{|b(n)|}}{n<\omega}$ does not converge to $0$ then $\balc_{b,h}=N$ where $N$ is the minimum natural number such that $B_N:=\set{n<\omega}{|b(n)|\leq N\cdot h(n)}$ is infinite, and $\dalc_{b,h}=\cfrak$.
\end{theorem}
\begin{proof}
  For all $n\in B_N$ there is a partition $\la c_{n,j}:\, j<N\ra$ of $b(n)$ such that $0<|c_{n,j}|\leq h(n)$ for each $j<N$. 
  Define $\Psi_-\colon N^{B_N}\to\Swf(b,h)$ such that, for each $z\in N^{B_N}$ and $n\in B_N$, $\Psi_-(z)(n)=c_{n,z(n)}$, and $\Psi_-(z)(i)=\emptyset$ for all $i\notin B_N$; and define $\Psi_+\colon \prod b\to N^{B_N}$ such that, for $n\in B_N$, $\Psi_+(x)(n)=j$ iff $x(n)\in c_{n,j}$. Note that the pair $(\Psi_-,\Psi_+)$ witnesses that $\Ed_N\leqT\aLc(b,h)$, so $\balc_{b,h}\leq\balc_{N,1}$ and $\dalc_{N,1}\leq\dalc_{b,h}$. On the other hand, by \autoref{aLc1} and~\autoref{Lochnotinfty}, $\balc_{N,1}=\blc_{N,N-1}=N$ and $\dalc_{N,1}=\dlc_{N,N-1}=\cfrak$. Therefore $\balc_{b,h}\leq N$ and $\cfrak\leq\dalc_{b,h}$.

  The converse inequality for $\balc_{b,h}$ follows from the fact that $(N-1)\cdot h(n)<|b(n)|$ for all but finitely many $n<\omega$ (so $N-1<\balc_{b,h}$).

  Put $b'(n):=\min\{\omega,|b(n)|\}$. As $|b'(n)|\leq|b(n)|$ for all $n<\omega$, by~\autoref{ex:easylc} $\dalc_{b,h}\leq\dalc_{b',h}$. On the other hand, $\dalc_{b',h}\leq\big|\prod b'\big|\leq\cfrak$, so $\dalc_{b,h}\leq \cfrak$.
\end{proof}

In conclusion, localization cardinals are interesting when $h(n)$ diverges to infinity, and anti-localization cardinals are interesting when the quotient $\frac{h(n)}{|b(n)|}$ converges to zero. In the particular case when $b$ is the constant function $\omega$, the parameter $h$ becomes irrelevant (more generally, $h$ is quite irrelevant when $b$ is as in \autoref{CM:unct}). To show this, we first present a couple of very useful preliminary results.

\begin{lemma}\label{poderLc}
    Let $b=\la b(i):\, i<\omega\ra$ be a sequence of non-empty sets and $h,h^+\in\omega^\omega$. Assume that there is some $I\in\Ior$ such that, for all $k<\omega$ and $i\in I_k$, $h(i)\geq h^+(k)$, and let $b^+(k):=\prod_{i\in I_k}b(i)$. Then $\Lc(b,h)\leqT \Lc(b^+,h^+)$.
\end{lemma}
\begin{proof}
    Define $\Psi_-\colon\prod b\to\prod b^+$ and $\Psi_+\colon \Swf(b^+,h^+)\to\Swf(b,h)$ by $\Psi_-(x)=\la x\frestr I_k:\, k<\omega\ra$ and $\Psi_+(S)(i)=\{t(i):\, t\in S(k)\}$ where $k<\omega$ is the unique such that $i\in I_k$. Note that $|\Psi_+(S)(i)|\leq |S(k)|\leq h^+(k)\leq h(i)$, so $\Psi_+(S)\in\Swf(b,h)$. It is routine to show that $\Psi_-(x)\in^* S$ implies $x\in^* \Psi_+(S)$.
\end{proof}

\begin{lemma}\label{poderaLc}
    Let $b=\la b(i):\, i<\omega\ra$ be a sequence of non-empty sets, $h^+\in\omega^\omega$ with $h^+\geq^*1$. Let $I = \la I_n:\, n<\omega\ra$ be the interval partition of $\omega$ such that $|I_n|=h^+(n)$ for all $n<\omega$, and define $b^+(n):=\prod_{k\in I_n}b(k)$. Then $\aLc(b^+,h^+)\leqT \Ed_b$.
\end{lemma}
\begin{proof}
    Define $F\colon\Swf(b^+,h^+)\to \prod b$ in the following way: if $\varphi\in\Swf(b^+,h^+)$ then find $\varphi^*\in\Swf(b^+,h^+)$ such that $\varphi(n)\subseteq\varphi^*(n)$ and $|\varphi^*(n)|=h^+(n)=|I_n|$ for all $n$, so enumerate $\varphi^*(n):=\{t^\varphi_{n,k}:\, k\in I_n\}$ and define $F(\varphi):=x_\varphi$ by $x_\varphi(k):=t^\varphi_{n,k}(k)$ when $k\in I_n$. On the other hand, define $G\colon \prod b\to\prod b^+$ by $G(y):=\la y\frestr I_n:\, n<\omega\ra$. It remains to show that $G(y)\in^\infty \varphi$ implies $y=^\infty x_\varphi$. There are infinitely many $n<\omega$ such that $y\frestr I_n\in\varphi(n)$. For such $n$, $y\frestr I_n=t^\varphi_{n,k}$ for some $k \in I_n$, so $x_\varphi(k)=t^\varphi_{n,k}(k)=y(k)$. Hence $x_\varphi =^\infty y$.
\end{proof}

\begin{theorem}\label{Lcomegah}
If $h,h'\in\omega^\omega$ are functions diverging to $\infty$, then $\Lc(\omega,h)\eqT\Lc(\omega,h')$. In particular, $\blc_{\omega,h}=\blc_{\omega,h'}$ and $\dlc_{\omega,h} = \dlc_{\omega,h'}$.
\end{theorem}
\begin{proof}
Fix $h^+\in\omega^\omega$ strictly increasing such that $h^+(0)=0$, $h^+\geq^* h$ and $h^+\geq^*h'$. It is enough to show that $\Lc(\omega,h^+)\eqT\Lc(\omega,h)$ (the same proof also works for $h'$). By \autoref{ex:easylc} we get $\Lc(\omega,h^+)\leqT\Lc(\omega,h)$.

For the converse,
define $g\colon\omega\to\omega$ strictly increasing such that $g(0)=0$ and 
$$\forall\, k<\omega\ \forall\, i\geq g(k)\colon h(i)\geq h^+(k).$$
Let $I_k=[g(k),g(k+1))$ (interval in $\omega$) and set $b^+(k):=\omega^{I_k}$, which has size $\aleph_0$. Hence $\Lc(b^+,h^+)\eqT\Lc(\omega,h^+)$ by \autoref{ex:easylc}, so we only need to show $\Lc(\omega,h)\leqT\Lc(b^+,h^+)$. But this follows by \autoref{poderLc}.
\end{proof}


\begin{theorem}\label{aLcEd}
If $h\geq^*1$ then $\aLc(\omega,h)\eqT\Ed$.
\end{theorem}
\begin{proof}
Recall from \autoref{aLc1} that $\Ed \eqT \aLc(\omega,1)$. It is clear that $\aLc(\omega,1)\leqT\aLc(\omega,h)$ follows by \autoref{ex:easylc}, so it is enough to show that $\aLc(\omega,h)\leqT\Ed$. Let $\la I_n:\, n<\omega\ra$ be the interval partition of $\omega$ such that $|I_n|=h(n)$, and let $b(n):=\omega^{I_n}$. Since $|b(n)|=\aleph_0$ for all but finitely many $n$, $\aLc(b,h)\eqT\aLc(\omega,h)$, so we only need to show $\aLc(b,h)\leqT\Ed$. But this follows by \autoref{poderaLc}.
\end{proof}

The previous results indicate that we only get four cardinals from $\blc_{\omega,h}$, $\dlc_{\omega,h}$, $\balc_{\omega,h}$ and $\dalc_{\omega,h}$ (the first pair when $h\to\infty$, and the second when $h\geq^*1$). We show in the following sections the pretty well-known result of Bartoszy\'nski that these cardinals characterize concrete cardinals associated with measure and category, concretely, the additivity and cofinality of measure, and the uniformity and covering of category.

To finish this section, we discuss the limits of localization and anti-localization cardinals.

\begin{definition}\label{DefminLc}
Define the following cardinal characteristics.
   \[\begin{array}{rlrl}
      \minLc_h:= & \min\{\blc_{b,h}:\, b\in\omega^\omega\}, &
      \supLc_h:= & \sup\{\dlc_{b,h}:\, b\in\omega^\omega\},\\
      \supaLc_h:= & \sup\{\balc_{b,h}:\, b\in\omega^\omega\}, &
      \minaLc_h:= & \min\{\dalc_{b,h}:\, b\in\omega^\omega\}.
   \end{array}\]
   For the first two cardinals we suppress $h$ when $h=\id_\omega$, and for the last two cardinals we suppress $h$ when $h=1$.
\end{definition}

The parameter $h$ is irrelevant for these limits.

\begin{theorem}\label{Lcsup}
Let $h\in\omega^\omega$. Then:
   \begin{enumerate}[label=\normalfont(\alph*)]
       \item\label{Lcsupone} $\minLc_h=\minLc$ and $\supLc_h=\supLc$ when $h$ goes to infinity.
       \item\label{Lcsuptwo} $\supaLc_h=\supaLc$ and $\minaLc_h =\minaLc$ when $h\geq^*1$.
   \end{enumerate}
\end{theorem}
\begin{proof}   
   \ref{Lcsupone}: Pick some $h^+\in\omega^\omega$ increasing such that $h^+(0)=0$ and $h^+\geq^* h$. Since $h$ diverges to infinity, we can find some $I\in\Ior$ as in \autoref{poderLc}, so there is some $b^+\in \omega^\omega$ such that $\Lc(b,h)\leqT \Lc(b^+,h^+)$. Therefore 
   $\minLc_{h^+}\leq\blc_{b^+,h^+}\leq\blc_{b,h}$ and $\dlc_{b,h}\leq\dlc_{b^+,h^+}\leq\supLc_{h^+}$ for any $b\in\omega$, so $\minLc_{h^+}\leq \minLc_h$ and $\supLc_h\leq \supLc_{h^+}$. The converse inequalities follow easily by~\autoref{ex:easylc}.

   Since $h^+\geq \id_\omega$, the same argument shows that $\minLc = \minLc_{h^+}$ and $\supLc = \supLc_{h^+}$.

   \ref{Lcsuptwo}: Let $b\in\omega^\omega$. By \autoref{poderaLc} we can find some $b'\in\omega^\omega$ such that $\aLc(b',h)\leqT \Ed_b$, so $\minaLc_h\leq \dalc_{b',h}\leq \dalc_{b,1}$ and $\balc_{b,1}\leq \balc_{b',h}\leq \supaLc_h$. Therefore, $\minaLc_h\leq \minaLc$ and $\supaLc\leq \supaLc_h$. The converse inequalities follow by \autoref{ex:easylc}.
\end{proof}

These cardinals can be used to characterize the localization and anti-localization cardinals for $b=\omega$.

\begin{theorem}\label{Lcomega}
   $\blc_{\omega,h}=\min\{\bfrak,\minLc_h\}$ and $\dlc_{\omega,h}=\max\{\dfrak,\supLc_h\}$.
\end{theorem}
\begin{proof}
    Assume that $F\subseteq\omega^\omega$ and $|F|<\min\{\bfrak,\minLc_h\}$. Therefore, there is some $d\in\omega^\omega$ such that, for every $x\in F$, $\forall^\infty\, i<\omega\colon x(i)<d(i)$. On the other hand, as $|F|<\blc_{d,h}$, we can find an slalom in $\Swf(d,h)\subseteq\Swf(\omega,h)$ that localizes all the reals in $F$ (just use a family $F'\subseteq\prod d$ of the same size as $F$ such that each member of $F'$ is a finite modification of some member of $F$ and vice-versa). Therefore, $\blc_{\omega,h}\geq\min\{\bfrak,\minLc_h\}$.

    Now we prove $\dlc_{\omega,h}\leq\max\{\dfrak,\supLc_h\}$. Choose a dominating family $D\subseteq\omega^\omega$ of size $\dfrak$ and, for each $d\in D$, choose  a family of slaloms $S_d\subseteq\Swf(d,h)$ that witnesses $\dlc_{d,h}$. Note that $S:=\bigcup_{d\in D}S_d\subseteq\Swf(\omega,h)$ has size $\leq\max\{\dfrak,\supLc_h\}$ and that every real in $\omega^\omega$ is localized by some slalom in $S$, so $\dlc_{\omega,h}\leq|S|$.

    It is easy to see that $\omega^\omega\leqT\Lc(\omega,h)$, so
    $\blc_{\omega,h}\leq\bfrak$ and $\dfrak\leq\dlc_{\omega,h}$. On the other hand, as $\blc_{\omega,h}\leq\blc_{b,h}$ and $\dlc_{b,h}\leq\dlc_{\omega,h}$ for any $b\in\omega^\omega$ (by~\autoref{ex:easylc}), the converse inequalities follow.
\end{proof}

Just as in~\autoref{Lcomega}, we have:

\begin{theorem}\label{aLcomega}
$\min\{\bfrak,\dalc_{\omega,h}\}=\min\{\bfrak,\minaLc_h\}$ and $\max\{\dfrak,\balc_{\omega,h}\}=\max\{\dfrak,\supaLc_h\}$.
\end{theorem}

Other cardinals like $\sup\{\blc_{b,h}:\, b\in\omega^\omega\}$ are in principle not that interesting: this supremum would be $\blc_{1,h}$, which is undefined, or at least if $b$ is restricted to be above $h$, then it would be $\blc_{h+1,h}=\balc_{h+1,1}$. Also, when $h$ does not go to infinity, $\minLc_h$ and $\supLc_h$ are easily characterized by~\autoref{Lochnotinfty}. However, when putting more restrictions, we have some interesting limits we will discuss in the following sections, for example, the minimum of $\balc_{b,h}$ for $b,h\in\omega^\omega$ such that $\sum_{n<\omega}\frac{h(n)}{b(n)}$ converges, has a deep connection with the covering of measure.

\section{Measure and localization}\label{sec:NLc}

We first review our terminology.
Given a sequence $b=\seq{b(n)}{n\in \omega}$ of non-empty sets, denote 
\[\Seq_{<\omega} b:=\bigcup_{n<\omega}\prod_{i<n}b(i).\] 
For each $\sigma\in\Seq_{<\omega}(b)$ define
\[[s]:=[s]_b:=\set{x\in\prod b}{ s\subseteq x}.\]
As a topological space, $\prod b$ has the product topology with each $b(n)$ endowed with the discrete topology. Note that $\set{[s]_b}{s\in \Seqw b}$ forms a base of clopen sets for this topology. When each $b(n)$ is countable we have that $\prod b$ is a Polish space and, in addition, if $|b(n)|\geq 2$ for infinitely many $n$, then $\prod b$ is a perfect Polish space. The most relevant instances are:
\begin{itemize}
   \item The Cantor space $2^\omega$, when $b(n)=2$ for all $n$, and 
   \item The Baire space $\omega^\omega$, when $b(n)=\omega$ for all $n$.
\end{itemize}
We now review the Lebesgue measure on $\prod b$ when each $b(n)\leq\omega$ is an ordinal. For any ordinal $0<\eta\leq\omega$, the probability measure $\mu_\eta$ on the power set of $\eta$ is defined by:
\begin{itemize}
\item when $\eta=n<\omega$, $\mu_{n}$ is the measure such that, for all $k<n$,
$\mu_{n}(\{k\})=\frac{1}{n}$, and 
\item when $\eta=\omega$, $\mu_{\omega}$ is the measure such that, for  $k<\omega$,
$\mu_{\omega}(\{k\})=\frac{1}{2^{k+1}}$.
\end{itemize}
Denote by $\Lb_b$ the product measure of $\la \mu_{b(n)}:\, n<\omega\ra$, which we call \emph{the Lebesgue Measure on $\prod b$}, so $\Lb_b$ is a  probability measure on the Borel $\sigma$-algebra of $\prod b$. More concretely, $\Lb_{b}$ 
is the unique measure on the Borel $\sigma$-algebra such that, for any $s\in \Seqw b$, $\Lb_b([s])=\prod_{i<|s|}\mu_{b(i)}(\{s(i)\})$. In particular, denote by $\Lb$, $\Lb_{2}$ and $\Lb_{\omega}$ the Lebesgue measure on $\R$, on $2^\omega$, and on $\omega^\omega$, respectively. 

\begin{remark}\label{TukeyPolish}
\ 
    \begin{enumerate}[label= \rm (\arabic*)]
        \item If $X$ is a perfect Polish space, then $\Mwf(X)\eqT \Mwf(\R)$ and $\Cbf_{\Mwf(X)}\eqT \Cbf_{\Mwf(\R)}$, where $\Mwf(X)$ denotes the ideal of meager subsets of $X$ (see~\cite[Subsec.~15.D]{Ke2}). Therefore, the cardinal characteristics associated with the meager ideal are independent of the perfect Polish space used to calculate it. When the space is clear from the context, we write $\Mwf$ for the meager ideal.
    
    \item Let $X$ be a Polish space. Denote by $\Bwf(X)$ the $\sigma$-algebra of Borel subsets of $X$, and assume that $\mu\colon\Bwf(X)\to [0,\infty]$ is a $\sigma$-finite measure such that $\mu(X)>0$ and every singleton has measure zero.
    Denote by $\Nwf(\mu)$ the ideal generated by the $\mu$-measure zero sets, which is also denoted by $\Nwf(X)$ when the measure on $X$ is clear.
    Then $\Nwf(\mu)\eqT \Nwf(\Lb)$ and $\Cbf_{\Nwf(\mu)}\eqT \Cbf_{\Nwf(\Lb)}$ where $\Lb$ is the Lebesgue measure on $\R$ (see~\cite[Thm.~17.41]{Ke2}). Therefore, the quadruples cardinal characteristics associated with both measure zero ideals are the same.

    When $b=\la b(n):\, n<\omega\ra$, each $b(n)\leq \omega$ is a non-zero ordinal, and $\prod b$ is perfect, we have that $\Lb_b$ satisfies the properties of $\mu$ above.

    When the measure space is understood, we just write $\Nwf$ for the null ideal.

    \item\label{TukeyPolishc} For $b$ as above, denote by $\Ewf(\prod b)$ the ideal generated by the $F_\sigma$ measure zero subsets of $\prod b$. Likewise, define $\Ewf(\R)$ and $\Ewf([0,1])$.
When $\prod b$ is perfect, the map $F_b\colon \prod b\to[0,1]$ defined by
\[F_b(x):=\sum_{n<\omega}\frac{x(n)}{\prod_{i\leq n}b(i)}\]
is a continuous onto function, and it preserves measure. Hence, this map preserves sets between $\Ewf(\prod b)$ and $\Ewf([0,1])$ via images and pre-images. Therefore, $\Ewf(\prod b)\eqT \Ewf([0,1])$ and $\Cbf_{\Ewf(\prod b)} \eqT \Cbf_{\Ewf([0,1])}$. We also have $\Ewf(\R)\eqT \Ewf([0,1])$ and $\Cbf_{\R} \eqT \Cbf_{\Ewf([0,1])}$, as well as $\Ewf(\omega^\omega)\eqT \Ewf(2^\omega)$ and $\Cbf_{\Ewf(\omega^\omega)}\eqT \Cbf_{\Ewf(2^\omega)}$.

When the space is clear, we just write $\Ewf$.
Therefore, the cardinal characteristics associated with $\Ewf$ do not depend on the previous spaces.
    \end{enumerate}
\end{remark}


The purpose of this section is to show the following characterization of the additivity and cofinality of the null ideal.

\begin{theorem}[Bartoszy\'nski~1984 \cite{Ba1984}]\label{loc-antloc}
Let $h\in\omega^\omega$. If $h$ diverges to infinity, then $\Lc(\omega,h)\eqT\Nwf$. In particular, $\blc_{\omega,h}=\add(\Nwf)$ and $\dlc_{\omega,h}=\cof(\Nwf)$.  
\end{theorem}

The proof of the theorem above will be split into two parts. In the first one we prove $\Lc(\omega,h)\leqT\Nwf(\baire)$ and the second part we prove $\Nwf(\cantor)\leqT\Lc(\omega,h)$. Thanks to \autoref{Lcomegah}, it is enough to show this for a single suitable $h$, so we do it for those satisfying $\sum_{i<\omega}\frac{1}{h(i)}<\infty$. Before entering into the proof of the first Tukey connection, we present the following lemma about the probabilistic independence of sets.

        

\begin{lemma}\label{lem:indp}
  Let $C$ be a countable set and $f\colon C\to[0,1]$. Then there are clopen sets $\la K_i:\, i\in C\ra$ in $\baire$ such that $\Lb_\omega(K_i)=f(i)$ and $\Lb_\omega\Big(\bigcap_{i\in F} K_i\Big)=\prod_{i\in F}\Lb_\omega(K_i)$ for all finite $F\subseteq C$.
\end{lemma}
\begin{proof} For each $i\in C$ there is some $x_i\in\cantor$ such that $f(i)=\sum_{k\in\omega}x_i(k) 2^{-(k+1)}$. 
Choose a one-to-one function $g\colon C\to \omega$ and, for $i\in C$, define
$K_i:=\{z\in\baire:\, x_i(z(g(i)))=1\}$.
To see that $\Lb_\omega(K_i)=f(i)$ let $S_i:=x_i^{-1}[\{1\}]$, so we get  $K_i=\bigcup_{k\in S_i}\{z\in\baire :\, z(g(i))=k\}$. Note that this is a disjoint union and that each $\{z\in\baire :\, z(g(i))=k\}$ has measure $2^{-(k+1)}$. We leave the proof of independence to the reader.
\end{proof}

We are ready to prove the first Tukey connection for \autoref{loc-antloc}.

\begin{lemma}\label{Bartone}
Assume that $\sum_{i\in\omega}\frac{1}{h(i)}<\infty$. Then $\Lc(\omega,h)\leqT\Nwf(\baire)$.
\end{lemma}
\begin{proof}
 By~\autoref{lem:indp}, there are $\la K_{i,j}:\, i,j\in\omega\ra$ clopen in $\baire$ such that $\Lb_\omega(K_{i,j})=\frac{1}{h(i)}$ and $\Lb_\omega\Big(\bigcap_{(i,j)\in F} K_{i,j}\Big)=\prod_{(i,j)\in F}\Lb_\omega(K_{i,j})$ for any finite $F\subseteq\omega\times\omega$.

Define $F \colon \baire\to\Nwf(\baire)$ by $F(x):=\bigcap_m\bigcup_{n\geq m}K_{n,x(n)}$. We claim that $F(x)$ has measure zero. Indeed,
\[
  \Lb_\omega(F(x))\leq\Lb_\omega\bigg(\bigcup_{n\geq m}K_{n,x(n)}\bigg)\leq\sum_{n\geq m}\Lb_\omega(K_{n,x(n)})=\sum_{n\geq m}\frac{1}{h(n)}\to 0 \text{ when $m\to\infty$.}
\]

Construct $G\colon \Nwf\to\Swf(\omega,h)$ as follows. Let $B\in\Nwf$. We use the following fact, for which we omit the proof.

\begin{fct}\label{Bartfct1}
  If $A\in\Bwf(\baire)$ and $\Lb_\omega(A)>0$ then there is some compact $K\subseteq A$
  such that $\Lb_\omega(K)>0$ and $U\cap K\neq\emptyset\imp\Lb_\omega(U\cap K)>0$ for any open $U\subseteq\baire$.
\end{fct}

Choose $K_B\subseteq\baire\menos B$ as in~\autoref{Bartfct1}. Let $T^B:=\{s\in\omega^{<\omega}:\, K_B\cap[s]\neq\emptyset\}$ and set, for $s\in T^B$,
$g^B_s(i):=\{j\in\omega:\, K_B\cap[s]\cap K_{i,j}=\emptyset\}$. Then, for any $s\in T^B$,
\begin{align*} 
   0 & <  \Lb_\omega(K_B\cap[s])\leq\Lb_\omega\left(\bigcap_{i\in\omega}\bigcap_{j\in g^B_s(i)}(\baire\menos K_{i,j})\right) & \\
    & = \prod_i\prod_{j\in g^B_s(i)}\Lb_\omega(\baire\menos K_{i,j}) & \text{by independence,}\\
    & =\prod_i\prod_{j\in g^B_s(i)}\Big(1-\frac{1}{h(i)}\Big)=\prod_i\Big(1-\frac{1}{h(i)}\Big)^{|g^B_s(i)|} & \\
    & \leq\prod_i e^{-\frac{|g^B_s(i)|}{h(i)}} & \text{(using $1+x\leq e^x$)}\\
    & =e^{-\sum_i \frac{|g^B_s(i)|}{h(i)}}. &
\end{align*}
Hence $0<e^{-\sum_i \frac{|g^B_s(i)|}{h(i)}}$, which implies $\sum_{i<\omega} \frac{|g^B_s(i)|}{h(i)}<\infty$.

Fix a sequence $\la \varepsilon_s:\, s\in T^B \ra$ of positive reals such that $\sum_{s\in T^B}\varepsilon_s<1$. 
Define $f^B\in\omega^{T^B}$ such that $\sum_{i\geq f^B(s)}\frac{|g^B_s(i)|}{h(i)}< \varepsilon_s$, and define $G(B):=\varphi^B$ by 
\[\varphi^B(i):=\bigcup\{g^B_s(i):\, f^B(s)\leq i,\ s\in T^B\}.\] 
Note that $\varphi^B\in \Swf(\omega,h)$ because
\[\begin{split}
  |\varphi^B(i)| & \leq\sum\{|g^B_s(i)|:\, f^B(s)\leq i,\ s\in T^B\} \\
   & \leq h(i)\sum\{\varepsilon_s:\, f^B(s)\leq i,\ s\in T^B\}
    \leq h(i).
\end{split}\]

We finally show that,
If $x\in\baire$, $B\in\Nwf$ and $F(x)\subseteq B$ then $x\in^* G(B)=\varphi^B$.

Since $K_B\cap B=\emptyset$, $K_B\cap F(x)=\emptyset=K_B\cap\bigcap_m\bigcup_{i\geq m}K_{i,x(i)}$. Hence $\bigcap_m\bigcup_{i\geq m}(K_B\cap K_{i,x(i)})=\emptyset$. Then,
by Baire Category Theorem, $\bigcup_{i\geq m}(K_B\cap K_{i,x(i)})$ is not dense in $K_B$ for some $m$,
so there is some $s\in T^B$ such that $[s]\cap\bigcup_{i\geq m}(K_B\cap K_{i,x(i)})=\emptyset$. Hence
$K_B\cap[s]\cap K_{i,x(i)}=\emptyset$ for all $i\geq m$, so
$x(i)\in g^B_s(i)\subseteq\varphi^B(i)$ for all $i\geq\max\{m,f^B(s)\}$.
\end{proof}

For the converse Tukey connection of \autoref{loc-antloc},
we need some results about the coding of $G_\delta$ measure zero subsets of $\prod b$, with all $0<b(n)\leq\omega$ ordinals, which we shall state as follows.

\begin{definition}\label{def:codenull}
  Denote $\cc_b:=[\Seqw b]^{<\aleph_0}$. For each $p\in\cc_b$ define $\cp(p):=\bigcup_{s\in p}[s]$, which is a clopen subset of $\Seqw b$.
  
  Define $\nc_b$ as the collection of all $\bar p=\la p_n:\, n<\omega\ra\in\cc_b^\omega$ such that $N(\bar p)$ has measure zero in $\prod b$, where
  \[N(\bar p):=\bigcap_{n<\omega}\bigcup_{i\geq n}\cp(p_i)=\set{x\in \prod b}{|\set{i<\omega}{ x\in\cp(p_i)}|=\aleph_0}.\]
\end{definition}

\begin{fct}\label{ex:codenull}
\ 
\begin{enumerate}[label=\rm(\alph*)]
    \item For $\bar p\in\cc_b^\omega$, $\bar p\in\nc_b$ iff $\forall\, \varepsilon\in \Q^+\ \exists\, N<\omega\ \forall\, n\geq N\colon \Lb_b\Big(\bigcup_{N\leq i<n}\cp(p_i)\Big)<\varepsilon$.
    \item The set $\nc_b$ is Borel in $\cc^\omega_b$, where the latter space has the product topology with $\cc_b$ discrete.
    \item\label{ex:codenulld} For $\bar p\in\cc_b^\omega$, if $\sum_{n<\omega}\Lb_b(\cp(p_n))$ converges, then $\bar p\in\nc_b$.
\end{enumerate}
\end{fct}

\begin{lemma}\label{codenull}
Let $\bar\varepsilon=\la\varepsilon_n:\, n<\omega\ra$ be a sequence of positive reals. For any $X\in\Nwf(\prod b)$ there is some $\bar p\in\nc_b$ such that $X\subseteq N(\bar p)$ and $\Lb_b(\cp(p_n))<\varepsilon_n$ for all $n<\omega$.
\end{lemma}
\begin{proof}
Wlog, by reducing the values $\varepsilon_n$, we may assume that $\sum_{n<\omega}\varepsilon_n$ converges. Let $X\subseteq \prod b$ be of measure zero, so, for each $m<\omega$, there is some $\la s_{m,i}:\, i<\omega\ra\in(\Seqw b)^\omega$ such that $X\subseteq \bigcup_{i<\omega}[s_{m,i}]$ and $\sum_{i<\omega}\Lb([s_{m,i}])<2^{-m}\varepsilon_m$. Then, find an increasing function $f_m\in\omega^\omega$ with $f_m(0)=0$ such that $\sum_{i\geq f_m(k)}\Lb([s_{m,i}])<2^{-(m+k)}\varepsilon_{m+k}$ for all $0<k<\omega$.

For each $n<\omega$ define $p_n:=\bigcup_{m\leq n}\{[s_{m,i}]:\, f_m(n-m)\leq i<f_m(n-m+1)\}$. It is routine to check that $\bar p:=\la p_n:\, n<\omega\ra$ is as desired, in particular, $\bar p\in\nc_b$ by \autoref{ex:codenull}~\ref{ex:codenulld}.
\end{proof}

Now we can complete the proof of \autoref{loc-antloc}.

\begin{lemma}
Assume that $\sum_{n<\omega}\frac{1}{h(n)}<\infty$. Then $\Nwf(2^\omega)\leqT\Lc(\omega,h)$. 
\end{lemma}
\begin{proof}

To define Tukey connection, consider 
\[b(n):=\bigg\{p\in\cc_2:\, \Lb_2(\cp(p))\leq\frac{1}{2^{n+1}(h(n)+1)}\bigg\},\] 
which has size $\aleph_0$. Hence $\Lc(b,h)\eqT\Lc(\omega,h)$. Define $\Psi_-\colon\Nwf\to\prod b$ where $\Psi_-(X)\in \nc_2$ is obtained as in \autoref{codenull}, and define $\Psi_+\colon\Swf(b,h)\to\Nwf$ such that, for $\varphi\in\Swf(b,h)$, $\Psi_+(\varphi):=N(\bar p^\varphi)$ where $p^\varphi_n:=\bigcup\varphi(n)$. Routine calculations show that $\Lb(\cp(p^\varphi_n))\leq 2^{-(n+1)}$, so $G(\varphi)\in\Nwf$ by \autoref{ex:codenull}~\ref{ex:codenulld}. It is clear that $\Psi_-(X)\in^*\varphi$ implies $X\subseteq \Psi_+(\varphi)$.
\end{proof}

The following is another characterization of $\add(\Nwf)$ and $\cof(\Nwf)$ in terms of Localization cardinals for $b\in\omega^\omega$, which is just a reformulation of \autoref{Lcomega}.

\begin{theorem}\label{ChAdd}
   $\add(\Nwf)=\min\{\bfrak,\minLc\}$ and $\cof(\Nwf)=\max\{\dfrak,\supLc\}$.
\end{theorem}




\section{Category and antilocalization}\label{MaLc}

In this section, we present the connections of anti-localization cardinals with the cardinal characteristics associated with the meager ideal of the reals.
The main result in this sense is the following theorem.

\begin{theorem}[{Miller~\cite{Mi1982}, Bartoszy\'nski~\cite{Ba1987}}]\label{BarMill}
If $h\geq^*1$, then $\balc_{\omega,h}=\non(\Mwf)$ and $\dalc_{\omega,h}=\cov(\Mwf)$.   
\end{theorem}

Before plunging into the details of the proof, we introduce the following relational system for combinatorial purposes.

\begin{definition}\label{def:Ed*}
Define the relational system $\Ed^*_b=\la\prod b,[\omega]^{\aleph_0}\times\prod b,\eqcirc\ra$ where $x\eqcirc(w,y)$ iff $x(n)\neq y(n)$ for all but finitely many $n\in w$. Write $\Ed^*:=\Ed^*_\omega$.
\end{definition}

\autoref{BarMill} is settled thanks to~\autoref{aLcEd} and the following result. 

\begin{theorem}\label{Edchar} 
$\big((\omega^\omega)^\perp;(\Ed^*)^\perp\big)^\perp\leqT\Cbf_\Mwf\leqT\Ed\eqT \Ed^*\leqT\omega^\omega$. In particular, 
\[\bfrak(\Ed)=\bfrak(\Ed^*)=\non(\Mwf) \text{ and } \dfrak(\Ed)=\dfrak(\Ed^*)=\cov(\Mwf).\]
\end{theorem}
We split the proof of~\autoref{Edchar} in two parts:
\begin{enumerate}[label=\normalfont(\faBolt$_{\arabic*}$)]
    \item\label{EdcharA} $\Cbf_\Mwf\leqT\Ed\eqT \Ed^*\leqT\omega^\omega$.
    
    \item\label{EdcharB} $\big((\omega^\omega)^\perp;(\Ed^*)^\perp\big)^\perp\leqT\Cbf_\Mwf$.
\end{enumerate}

As a consequence of~\ref{EdcharA}, \[\bfrak\leq\bfrak(\Ed)=\bfrak(\Ed^*)\leq\non(\Mwf)\text{ and } \cov(\Mwf)\leq \dfrak(\Ed)=\dfrak(\Ed^*)\leq\dfrak.\]
On the other hand, let $\Rbf_0:=(\omega^\omega)^\perp$, $\Rbf_1:=(\Ed^*)^\perp$ and $\Rbf_*:=(\Rbf_0;\Rbf_1)$. Then, by employing~\autoref{blascomp},
\[
\begin{split}
    \dfrak(\Rbf_*^\perp) & = \bfrak(\Rbf_*)=\min\{\bfrak(\Rbf_0),\bfrak(\Rbf_1)\}=\min\{\dfrak,\dfrak(\Ed^*)\} \text{ and}\\ 
    \bfrak(\Rbf_*^\perp) & = \dfrak(\Rbf_*)=\dfrak(\Rbf_0)\cdot\dfrak(\Rbf_1)=\bfrak\cdot\bfrak(\Ed^*).
\end{split}
\]
This implies that $\dfrak(\Rbf_*^\perp)= \dfrak(\Ed^*)$ and $\bfrak(\Rbf_*^\perp)= \bfrak(\Ed^*)$. Therefore,~\ref{EdcharB} implies  $\dfrak(\Ed^*)\leq\cov(\Mwf)$ and $\non(\Mwf)\leq  \bfrak(\Ed^*)$, which establishes~\autoref{BarMill}. 

We now prove~\ref{EdcharA}. Denote by $\Fn(A,B)$ the set of finite partial functions from $A$ into $B$.

\begin{lemma}
$\Cbf_\Mwf\leqT\Ed\eqT \Ed^*\leqT\omega^\omega$
\end{lemma}
\begin{proof}
 For $\Cbf_\Mwf\leqT\Ed$, work in $\omega^\omega$ and use the maps $x\mapsto x$ and $y\mapsto \{x\in\omega^\omega:\, x\neq^\infty y\}$ (this is a meager $F_\sigma$-set). The relation $\Ed^*\leqT \omega^\omega$ is also simple using the maps $x\mapsto x$ and $y\mapsto (\omega,y+1)$ (where $y+1$ is interpreted as coordinate-wise sum).

Now we show $\Ed\eqT\Ed^*$. The relation $\Ed^*\leqT\Ed$ is easy. For the converse, let $b(n):=\{s\in\Fn(\omega,\omega):\, |s|=n+1\}$. Since $|b(n)|=\aleph_0$ for all $n<\omega$, $\Ed_b\eqT\Ed$, so we show $\Ed_b\leqT\Ed^*$. Define $G\colon [\omega]^{\aleph_0}\times\omega^\omega\to \prod b$ by $G(w,y):=\big\la y\frestr\{k^w_i:\, i\leq n\}:\, n<\omega\big\ra$ where $w:=\{k^w_i:\, i<\omega\}$ is the increasing enumeration. The function $F\colon\prod b\to \omega^\omega$ is defined as follows: for $x\in\prod b$ define a sequence $\la m^x_j:\, j<\omega\ra$ such that $m^x_j\in\dom x_j\menos\{m^x_i:\, i<j\}$ (which is fine because $|\dom x_j|=j+1$), and let $F(x):=F_x\in\omega^\omega$ be any function such that $F_x(m^x_j)=x_j(m^x_j)$ for all $j<\omega$. It remains to show that $x=^\infty G(w,y)$ implies that $F_x(n)=y(n)$ for infinitely many $n\in w$. Note that $x=^\infty G(w,y)$ means that $x_j=y\frestr\{k^w_i:\, i\leq j\}$ for infinitely many $j$. For these $j$, the $m^x_j$ are pairwise different members of $w$ and $F_x(m^x_j)=x_j(m^x_j)=y(m^x_j)$, so $F_x$ and $y$ are equal at infinitely many points of $w$.   
\end{proof}

To prove~\ref{EdcharB} we must work a bit more. To this end, we use the following characterization of meager sets. 

\begin{lemma}\label{aboutM}
Define $b^*(n):=\bigcup_{k>n}2^{[n,k)}$ and $b^*:=\la b^*(n):\, n<\omega\ra$. Then:
\begin{enumerate}[label=\rm(\alph*)]
    \item For any $f\in\prod b^*$, $M^*_f:=\set{x\in 2^\omega}{|\set{n<\omega}{x\supseteq f(n)}|<\aleph_0}$ is $F_\sigma$ and meager.
    \item If $X\in\Mwf(2^\omega)$ then $X\subseteq M^*_f$ for some $f\in\prod b^*$.
    \item $\Mg^*\eqT\Cbf_\Mwf$ where $\Mg^*:=\la 2^\omega,\prod b^*,\sqsubset\ra$ and $x\sqsubset f$ is defined by $x\in M^*_f$.
\end{enumerate}
\end{lemma}

Since each $b^*(n)$ is infinite countable, we clearly have:

\begin{lemma}\label{ex:othercodemg}
Let $b^*$ as in~\autoref{aboutM}. Then $\Ed^*_{b^*}\eqT\Ed^*$.
\end{lemma}

The following lemma will also be useful.

\begin{lemma}\label{otherD}
The relational system $\Dbf:=\la W^*,\omega^\omega,\leq^+ \ra$ is Tukey equivalent with $\omega^\omega$, where $W^*$ is the collection of increasing functions in $\omega^\omega$ and $x\leq^+ y$ iff $x_{k+1}\leq y(x_k)$ for all but finitely many $k$.
\end{lemma}
\begin{proof}
We use the relational systems from \autoref{Charb-d} and the notation introduced there. We prove that $\Dbf\leqT \Ior$ and $\Dbf_2\leqT \Dbf$.

The relation $\Dbf\leqT \Ior$ is witnessed by the maps $x\mapsto I_x$, where $I_x$ is an interval partition with endpoints $\{0\}\cup\set{x_k}{k<\omega}$, and $J\mapsto y_J$, where $y_J(k):= \min J_{n+2}$ whenever $k\in J_n$.

On the other hand, $\Dbf_2\leqT \Dbf$ is witnessed by $I\mapsto f_I$ and the map $F\colon \omega^\omega\to \Ior$ defined as follows: for $y\in\omega^\omega$, choose some increasing $y'\in\omega^\omega$ such that $y'(0)=0$ and $y'(n)>y(n)$ for all $n>0$, and define $F(y):= I^{y'}$. If $I \in \Ior$, $y\in\omega^\omega$ and $f_I\leq^+ y$ then, for eventually many $k$: find some $n$ such that $f_I(n)\leq {y'}^*(k) < f_I(n+1)$, so $f_I(n+1) \leq y(f_I(n)) < y'(f_I(n)) < y'({y'}^*(k)+1) = {y'}^*(k+1)$, which indicates that $I^{y'}_k$ is not contained in $I_n$, i.e.\ $I \ntriangleright I^{y'}$.
\end{proof}

We now have everything required for proving~\ref{EdcharB}.

\begin{lemma}
$\big((\omega^\omega)^\perp;(\Ed^*)^\perp\big)^\perp\leqT\Cbf_\Mwf$
\end{lemma}
\begin{proof}
By applying~\autoref{ex:compleqT},~\autoref{aboutM}, ~\ref{ex:othercodemg} and \ref{otherD}, it is enough to show $\Rbf^\perp_*\leqT\Mg^*$ where $\Rbf_*:=(\Rbf_0;\Rbf_1)$, $\Rbf_0:=\Dbf^\perp$ and $\Rbf_1:=(\Ed^*_{b^*})^\perp$ with $b^*$ as in \autoref{ex:othercodemg}. So $\Rbf_*^\perp=\la W^*\times\prod b^*, \omega^\omega\times([\omega]^{\aleph_0}\times\prod b^*)^{W^*}, R^\perp_* \ra$ where
\[(x,a)R_*^\perp(y,f)\text{ iff either } x\leq^+ y \text{ or } a\eqcirc f(x).\]

Given $a\in\prod b^*$, let $g_a\colon \omega\to\omega$ be such that $g_a(n)>n$ and $a(n)\in 2^{[n,g_a(n))}$.
Define $F\colon W^*\times\prod b^*\to 2^\omega$ such that, for $(x,a)\in W^*\times\prod b^*$, if $w_{x,a}:=\{k<\omega:\, g_a(x_k)<x_{k+1}\}$ is infinite (i.e.\ $x\nleq^+ g_a$) then $F(x,a)\supseteq a(x_k)$ for all $k\in w_{x,a}$.

Define $G\colon\prod b^*\to \omega^\omega\times([\omega]^{\aleph_0}\times\prod b^*)^{W^*}$ as follows. For $c\in\prod b^*$, define $f_c\colon W^*\to [\omega]^{\aleph_0}\times\prod b^*$ such that, whenever $x\nleq^+ g_c$, $f_c(x):=(z_{x,c},c)$ where $z_{x,c}:=\{x_k:\, k\in w_{x,c}\}$. So let $G(c):=(g_c,f_c)$.

We show that $\neg((x,a)R^\perp_*(g_c,f_c))$ implies $F(x,a)\notin M^*_c$. The hypothesis says that $x\nleq^+ g_c$ and that $w:=\{k\in w_{x,c}:\, a(x_k)=c(x_k)\}$ is infinite. For $k\in w$, $g_a(x_k)=g_c(x_k)$, so $w\subseteq w_{x,a}$ and $F(x,a)\supseteq a(x_k)=c(x_k)$ for all $k\in w$. This implies that $F(x,a)\notin M^*_c$.
\end{proof}

Using the equalities $\add(\Mwf)=\min\{\bfrak,\cov(\Mwf)\}$ and $\cof(\Mwf) = \max\{\dfrak,\non(\Mwf)\}$, we have the following reformulation of \autoref{aLcomega}.

\begin{theorem}
$\add(\Mwf)=\min\{\bfrak,\minaLc\}$ and $\cof(\Mwf)=\max\{\dfrak,\supaLc\}$   
\end{theorem}

\begin{remark}
    Let $\SNwf$ be the ideal of strong measure zero subsets of $\R$. Miller~\cite{Mi1981} proved that $\non(\SNwf) = \minaLc$ (which is also clear by using Yorioka ideals, see e.g. \cite{CM}). Hence, the previous theorem indicates that $\add(\Mwf) = \min\{\bfrak, \non(\SNwf)\}$.
\end{remark}

\section{Measure and antilocalization}\label{NaLc}

We present the following characterizations of the covering and uniformity of the null ideal in terms of anti-localization cardinals. The proofs rely on the combinatorics of small sets developed by Bartoszynski in~\cite{B1988} to study the cofinality of the covering of the null ideal.



\begin{theorem}[{Bartoszy\'nki 1988~\cite{B1988}}]\label{alcCh}
\begin{multline*}
\tag{A} \min\left(\{\bfrak\}\cup\set{\balc_{b, h}}{b, h\in\omega^\omega\text{ and }\sum_{n\in\omega}\frac{h(n)}{b(n)}<\infty}\right)\leq\cov(\Nwf)\\
 \leq\min\set{\balc_{b, h}}{b,h\in\omega^\omega\text{ and }\sum_{n\in\omega}\frac{h(n)}{b(n)}<\infty}.\label{alcChA}   
\end{multline*}
In particular, 
if $\cov(\Nwf)<\bfrak$ then \[\cov(\Nwf)=\min\set{\balc_{b, h}}{b,h\in\omega^\omega\text{ and }\sum_{n\in\omega}\frac{h(n)}{b(n)}<\infty}.\] 

 \begin{multline*}\label{alcChB}
\tag{B}\sup\set{\dalc_{b, h}}{b, h\in\omega^\omega\text{ and }\sum_{n\in\omega}\frac{h(n)}{b(n)}<\infty} \leq\non(\Nwf)\leq\\
  \sup\left(\{\dfrak\}\cup\set{\dalc_{b, h}}{b, h\in\omega^\omega\text{ and }\sum_{n\in\omega}\frac{h(n)}{b(n)}<\infty}\right).   
\end{multline*}
In particular, if $\non(\Nwf)>\dfrak$ then
    \[\non(\Nwf)=\sup\set{\dalc_{b, h}}{b, h\in\omega^\omega\text{ and }\sum_{n\in\omega}\frac{h(n)}{b(n)}<\infty}.\]  
\end{theorem}

Some of the inequalities can be deduced from the following result.

\begin{lemma}[{\cite[Lemma~2.3]{KM}}]\label{Tukalcov} Let $b, h\in\omega^\omega$ and assume that $b\geq 1$ and $h\geq^*1$.
\begin{enumerate}[label=\normalfont(\alph*)]
    \item\label{Tukone} If $\sum_{n\in\omega}\frac{h(n)}{b(n)}<\infty$, then $\Cbf_\Nwf\leqT\aLc(b,h)^\perp$. In particular $\cov(\Nwf)\leq\balc_{b,h}$ and $\dalc_{b,h}\leq \non(\Nwf)$.
    
    \item\label{Tuktwo} If $\sum_{n\in\omega}\frac{h(n)}{b(n)}=\infty$, then $\Cbf_\Ewf\leqT\aLc(b,h)$. In particular $\cov(\Ewf) \leq \dalc_{b,h}$ and $\balc_{b,h} \leq \non(\Ewf)$.
\end{enumerate}
\end{lemma}
\begin{proof}
\noindent\ref{Tukone}: It suffices to find functions $\Psi_-\colon \prod b\to\prod b$ and $\Psi_+\colon \Swf(b,h)\to\Nwf(\prod b)$ such that, for any $x\in\prod b$ and for any $\varphi\in\Swf(b,h)$, $\Psi_-(x)\in^\infty\varphi$ implies $x\in\Psi_+(\varphi)$. Define $\Psi_-\colon \prod b\to\prod b$ as the identity map and define $\Psi_+\colon \Swf(b,h)\to\Nwf(\prod b)$ by $\Psi_+(\varphi):=\set{x\in\prod b}{x \in^\infty \varphi}$. It is clear that $\Psi_-(x)\in^\infty\varphi$ implies $x\in\Psi_+(\varphi)$, but we must show that $\Psi_+(\varphi)\in\Nwf(\prod b)$, which in turn will guarantee that $(\Psi_-,\Psi_+)$ is the desired Tukey connection. Note that \[\Psi_+(\varphi)=\bigcap_{n\in\omega}\bigcup_{m\geq n}\set{x\in\prod b}{x(n)\in\varphi(n)}.\]
Now, for $n\in\omega$
\begin{align*}
\Lb_b(\Psi_+(\varphi))&=\lim_{n\to\infty}\Lb_b \left(\bigcup_{m\geq n}\set{x\in\prod b}{x(n)\in\varphi(n)}\right)\\
&\leq\lim_{n\to\infty}\sum_{m\geq n}\Lb_b\left(\set{x\in\prod b}{x(n)\in\varphi(n)}\right) \\
&\leq\lim_{n\to\infty}\sum_{m\geq n}\frac{|\varphi(n)|}{b(n)}\leq\lim_{n\to\infty}\sum_{m\geq n}\frac{h(n)}{b(n)}=0,
\end{align*}
 where the last equality hold because $\sum_{n\in\omega}\frac{h(n)}{b(n)}<\infty$.
 
\noindent\ref{Tuktwo}: By~\autoref{aLc1} and~\autoref{Ed-aLc} it is enough to prove $\Cwf_{\Ewf}\leqT\Ed_d$ whenever $\sum_{n<\omega}\frac{1}{d(n)}=\infty$. 
The case $d\leq^*1$ is trivial, so we can assume $d\not\leq^*1$. For $y\in\prod d$, set $\Psi_+(y):=\set{x\in\prod d}{x \neq^\infty y}$. Notice that $\Lb_d(\Psi_+(y))=\lim_{n\to\infty}\prod_{m\geq n}\bigg(1-\frac{1}{d(m)}\bigg)$ and by employing the fact that $1+x\leq e^x$ for any $x$, we obtain that,
\[\prod_{m\geq n}\bigg(1-\frac{1}{d(m)}\bigg)\leq\prod_{m\geq n}e^{-\frac{1}{d(m)}}=e^{-\sum_{m\leq n}\frac{1}{d(m)}}=0,\]
hence $\Psi_+(x)\in\Ewf(\prod d)$. By setting $\Psi_-\colon \prod d\to\prod d$ as the identity function, $(\Psi_-,\Psi_+)$ is the desired Tukey connection.  
\end{proof}

We review the notion of \emph{small sets} from Bartoszy\'nski.

\newcommand{\twosmall}{2-small coding}

\begin{definition}[{\cite{B1988}}]\label{def:small} 
Let $I=\seq{I_n}{n\in\omega}$ be a sequence of finite subsets of $\omega$. For $\varphi\in\prod_{n<\omega}\pts(2^{I_n})$, we define
\[N_\varphi:=\set{x\in2^\omega}{\exists^{\infty}\, n\in\omega\colon x{\upharpoonright}I_n\in \varphi(n)}.\]
Note that $N_\varphi\in\Nwf$ whenever $\sum_{n<\omega}\frac{|\varphi(n)|}{2^{|I_n|}}<\infty$ (by \autoref{ex:codenull}~\ref{ex:codenulld}).

A set $X\subseteq2^\omega$ is \textit{small} if there is a partition $I=\seq{I_n}{n\in\omega}$ of a cofinite subset of $\omega$ into finite sets and $\varphi\in\prod_{n<\omega}\pts(2^{I_n})$ such that $X\subseteq N_\varphi$ and  $\sum_{n<\omega}\frac{|\varphi(n)|}{2^{|J_n|}}<\infty$.
    
    
    
    

\end{definition}

We can reformulate \autoref{codenull} as follows:

\begin{lemma}\label{presmall}
    Assume that $X\in\Nwf(2^\omega)$. Then there is some $\varphi\in\prod_{k<\omega}\pts(2^k)$ such that $\sum_{k<\omega}\frac{|\varphi(k)|}{2^k}<\infty$ and $X\subseteq N_\varphi$.
\end{lemma}
\begin{proof}
    Let $\bar \varepsilon =\la \varepsilon_n:\, n<\omega\ra$ be a sequence of positive reals such that $\sum_{n<\omega}\varepsilon_n<\infty$. By \autoref{codenull}, there is some $\bar p\in \nc_2$ such that $X\subseteq N(\bar p)$ and $\Lb_2(\cc(p_n))<\varepsilon_n$ for all $n<\omega$. Since each $\cc(p_n)$ is a clopen subset of $2^\omega$, we can find an increasing sequence $\la k_n:\, n<\omega\ra$ of natural numbers and, for each $n<\omega$, a set $T_n\subseteq 2^{k_n}$, such that $\cc(p_n)=\set{x\in 2^\omega}{x\frestr k_n\in T_n}$. Note that $\Lb_2(\cc(p_n))=\frac{|T_n|}{2^{k_n}}<\varepsilon_n$.
    
    Define
    \[\varphi(k):=\left\{
      \begin{array}{ll}
          T_n &  \text{if $k=k_n$ for some $n<\omega$,}\\
          \emptyset & \text{otherwise.}
      \end{array}
    \right.\]
    So $\varphi$ is as required.
\end{proof}

In spite of the previous result, Bartoszy\'nski~\cite{B1988} proved that there are measure zero sets that are not small. He showed, on the other hand, that every null sets can be covered by two small sets. This inspires the following definition.


\begin{definition}\label{def:2small}
We say that $\tbf=(L,\varphi,\varphi')$ is \emph{\twosmall} if:
\begin{enumerate}[label =\normalfont (T\arabic*)]
    \item\label{def:2smalla} $ L\in\Ior$, and we denote $L_{2k}=[n_k,m_k)$ and $L_{2k+1}=[m_k,n_{k+1})$ (so $n_k < m_k <n_{k+1}$), and define $I:=\la I_k:\, k<\omega\ra$ and $I':=\la I'_k:\, k<\omega\ra$ by $I_k:=[n_k,n_{k+1})$ and $I'_k:=[m_k,m_{k+1})$.
    \item\label{def:2smallb} $\varphi\in \prod_{n<\omega}\pts(2^{I_n})$ and $\varphi'\in\prod_{n<\omega}\pts(2^{I'_n})$ satisfies $\frac{|\varphi(k)|}{2^{|I_k|}}<\frac{1}{4^{k+1}}$ and $\frac{|\varphi'(k)|}{2^{|I'_k|}}<\frac{1}{4^{k+1}}$.
\end{enumerate}
\end{definition}

The following result, due to Bartoszy\'nski, not only shows that any null set can be covered by two small sets, but that less than $\bfrak$ many null sets can be contained in unions of pairs of small sets that come from the same partition $L$:

\begin{lemma}\label{U2small}
 Let $\set{G_\alpha}{\alpha<\kappa}$ be a family of sets in $\Nwf(2^\omega)$ and let $D\subseteq\Ior$ be dominating on $\Ior$. If $\kappa<\bfrak$, then there is some $L\in D$ and some \twosmall s $\tbf^\alpha=(L,\varphi_\alpha,\varphi'_\alpha)$ (on the same partition $L$) for $\alpha<\kappa$ such that, for any $\alpha<\kappa$, $G_\alpha \subseteq N_{\varphi_\alpha}\cup N_{\varphi'_\alpha}$.
 

\end{lemma}
\begin{proof}
    For each $\alpha<\kappa$ find some $\psi_\alpha\in\prod_{k<\omega}\pts(2^k)$ such that $\sum_{k<\omega}\frac{|\psi_\alpha(k)|}{2^k}<\infty$ and $G_\alpha\subseteq N_{\psi_\alpha}$ (by \autoref{presmall}). By recursion on $k<\omega$ define $n^\alpha_k$ and $m^\alpha_k$ such that $n^\alpha_0=0$,
    \[
    \begin{split}
     m^\alpha_k & :=\min\set{j<\omega}{j>n_k,\ 2^{n^\alpha_k}\sum_{i\geq j}\frac{|\psi_\alpha(i)|}{2^i}<\frac{1}{4^{k+1}}}, \text{ and }\\
     n^\alpha_{k+1} & :=\min\set{j<\omega}{j>m_k,\ 2^{m^\alpha_k}\sum_{i\geq j}\frac{|\psi_\alpha(i)|}{2^i}<\frac{1}{4^{k+1}}}.
    \end{split}
    \]
    Define $L^\alpha\in\Ior$ by $L^\alpha_k:=[n^\alpha_{2k},n^\alpha_{2(k+1)})$. Since $\kappa<\bfrak=\bfrak(\Ior)$, find some $L\in D$ such that $L^\alpha \sqsubseteq L$ for all $\alpha<\kappa$ (see \autoref{Charb-d}). Denote $L_{2k}=[n_k,m_k)$ and $L_{2k+1} = [m_k, n_{k+1})$, and define $I_k:=[n_k,n_{k+1})$ and $I'_k:=[m_k,m_{k+1})$. Then, for each $\alpha<\kappa$, there is some $j_\alpha<\omega$ such that, for all $k\geq j_\alpha$, $L_{2k}$ contains $[n^\alpha_\ell,m^\alpha_\ell)$ and $L_{2k+1}$ contains $[m^\alpha_{\ell'}, n^\alpha_{\ell'+1})$ for some $\ell,\ell'\geq k$.
    
    For each $\alpha<\kappa$, define $\varphi_\alpha$ and $\varphi'_\alpha$ by
    \[
    \begin{split}
        \varphi_\alpha(k) & :=\set{s\in 2^{I_k}}{s \text{ is compatible with some } t\in \bigcup_{i=m_k}^{n_{k+1}-1}\psi_\alpha(i)}, \text{ and}\\
        \varphi'_\alpha(k) & :=\set{s\in 2^{I'_k}}{s \text{ is compatible with some } t\in \bigcup_{i=n_{k+1}}^{m_{k+1}-1}\psi_\alpha(i)}
    \end{split}    
    \]
    whenever $k\geq j_\alpha$, otherwise $ \varphi_\alpha(k) =  \varphi'_\alpha(k) =\emptyset$.
    
    We show that $\tbf_\alpha:=(L,\varphi_\alpha, \varphi'_\alpha)$ is as required, i.e.\ we show that~\ref{def:2smallb} of \autoref{def:2small} is satisfied, and that $G_\alpha\subseteq N_{\varphi_\alpha}\cup N_{\varphi'_\alpha}$. The property~\ref{def:2smallb} is clear when $k<j_\alpha$, and when $k\geq j_\alpha$, $n_k\leq n^\alpha_\ell < m^\alpha_\ell \leq m_k$ and $m_k \leq m^\alpha_{\ell'}< n^\alpha_{\ell'+1}\leq n_{k+1}$ for some $\ell,\ell'\geq k$, so
    \[\frac{|\varphi_\alpha(k)|}{2^{|I_k|}}\leq 2^{n_k}\sum_{i\geq m_k}\frac{|\psi_\alpha(i)|}{2^i} \leq 2^{n^\alpha_\ell}\sum_{i\geq m^\alpha_\ell}\frac{|\psi_\alpha(i)|}{2^i}<\frac{1}{4^{\ell +1}}\leq \frac{1}{4^{k +1}}\]
    and, similarly, \[\frac{|\varphi'_\alpha(k)|}{2^{|I'_k|}} < \frac{1}{4^{k +1}}.\]

    For $G_\alpha\subseteq N_{\varphi_\alpha}\cup N_{\varphi'_\alpha}$: if $x\in G_\alpha$ then the set $K:=\set{i<\omega}{x\frestr i \in \psi_\alpha(i)}$ is infinite, so one of $K_1:=K\cap \bigcup_{k<\omega}[m_k,n_{k+1})$ and $K_2:=K\cap \bigcup_{k<\omega}[n_{k+1},m_{k+1})$ is infinite. In case $K_1$ is infinite, we have that $x\frestr I_k\in \varphi_\alpha(k)$ for all $k$ such that $K\cap[m_k,n_{k+1})\neq \emptyset$, and there are infinitely many such $k$, so $x\in N_{\varphi_\alpha}$; the case that $K_2$ is infinite allows to conclude, in a similar way, that $x\in N_{\varphi'_\alpha}$.
\end{proof}

The result below is proven by unraveling Bartoszy\'nski's original proof of~\autoref{alcCh}. 

\begin{lemma}\label{U2smallmain}
    Assume that $\tbf=(L,\varphi, \varphi')$ is a \twosmall. Define, for $k<\omega$,
    \begin{align*}
    S_k^\tbf & :=\set{s\in2^{[n_k,m_k)}}{s\text{\ has at least $2^{n_{k+1}-m_k-k}$ many extensions inside $\varphi(k)$}},\\
   \tilde S_k^\tbf & :=\set{s\in2^{[n_k,m_k)}}{s\text{\ has at least $2^{n_{k}-m_{k-1}-k}$ many extensions inside $\varphi'(k-1)$}},
\end{align*}
and set $\tilde S^\tbf_0:=\emptyset$. Then:
\begin{enumerate}[label = \normalfont (\alph*)]
    \item\label{U2smallmaina} $\frac{|S_k^\tbf|}{2^{m_k-n_k}}< \frac{1}{2^{k+2}}$ and $\frac{|\tilde S_k^\tbf|}{2^{m_k-n_k}}< \frac{1}{2^{k}}$ for all $k<\omega$.
    \item\label{U2smallmainb} The slalom $\varphi_\tbf\in\prod_{k<\omega}2^{I_k}$ defined by $\varphi_\tbf(k):=\set{s\in2^{I_{k}}}{s{\upharpoonright}[n_k,m_k)\in S_k^\tbf\cup \tilde S_k^\tbf}$ satisfies $\frac{|\varphi_\tbf(k)|}{2^{|I_k|}} < \frac{1}{2^{k-1}}$.
    \item\label{U2smallmainc} Assume $x\in 2^\omega$ such that $\exists\, k_0<\omega\ \forall\, k\geq k_0\colon x\frestr I_k\notin \varphi_\tbf(k)$. Denote $\tilde I_k:= [m_k, n_{k+1})$ for $k<\omega$, and define:
    \[\psi_\tbf^x(k):=\set{s\in2^{\tilde I_k}}{x{\upharpoonright}[n_k,m_k){}^{\frown} s\in \varphi(k)\text{\ or\ }s{}^{\frown}x{\upharpoonright}[n_{k+1},m_{k+1})\in \varphi'(k)}\]
    whenever $k\geq k_0$, or $\psi^x_\tbf(k):=\emptyset$ otherwise.
    Then $\forall\, k<\omega\colon \frac{|\psi_\tbf^x(k)|}{2^{|\tilde I_k|}}< \frac{1}{2^{k-1}}$.

    \item\label{U2smallmaind} Assume $x\in 2^\omega$ is as in the previous item and $y\in 2^\omega$ satisfies $\forall^\infty\, k<\omega\colon y\frestr \tilde I_k \notin \psi^x_\tbf(k)$. Define $z_{x,y}\in 2^\omega$ by
    \[z_{x,y}(i):=\left\{
      \begin{array}{ll}
         y(i)  & \text{if $i\in \tilde I_k$ for some $k$,}\\
         x(i)  & \text{otherwise.}
      \end{array}\right.\]
      Then $z_{x,y}\notin N_\varphi\cup N_{\varphi'}$.
\end{enumerate}
\end{lemma}
\begin{proof}
\noindent\ref{U2smallmaina}: We just estimate the size of $S_k^\tbf$. For $k\in\omega$, observe that 
\[|S_k^\tbf|\cdot2^{n_{k+1}-m_{k}-k}\leq|\varphi(k)|.\]
Therefore,
\[|S_k^\tbf|\leq2^{k}\cdot\frac{|\varphi(k)|}{2^{n_{k+1}-n_{k}}}\cdot2^{m_{k}-n_{k}},\]
which implies by~\autoref{def:2small}~\ref{def:2smallb} that
\[\frac{|S_k^\tbf|}{2^{m_{k}-n_{k}}}\leq2^k\cdot\frac{|\varphi(k)|}{2^{n_{k+1}-n_{k}}}<\frac{1}{2^{k+2}}.\]
In a similar way, we can show that
\[\frac{|\tilde S_k^\tbf|}{2^{m_{k}-n_{k}}} < \frac{1}{2^{k}}.\]

\noindent\ref{U2smallmainb}: From~\ref{U2smallmaina} because $\frac{\varphi_\tbf(k)}{2^{|I_k|}}\leq \frac{|S_k^\tbf|}{2^{m_k-n_k}} + \frac{|\tilde S_k^\tbf|}{2^{m_k-n_k}}$.

\noindent\ref{U2smallmainc}: Notice that for all $k\geq k_0$,
\[\frac{|\psi_\tbf^x(k)|}{2^{|\tilde I_k|}}=\frac{|\psi_\tbf^x(k)|}{2^{n_{k+1}-m_k}}< \frac{2^{n_{k+1}-m_k-k}}{2^{n_{k+1}-m_k}}+\frac{2^{n_{k+1}-m_k-(k+1)}}{2^{n_{k+1}-m_k}}<\frac{1}{2^{k-1}}.\]

\noindent\ref{U2smallmaind}: Let $z:=z_{x,y}$. The assumption states that, for all but finitely many $k<\omega$, $x\frestr [n_k,m_k){}^\frown y\frestr [m_k,n_{k+1}\notin \varphi(k)$ and $y\frestr [m_k,n_{k+1}){}^\frown x\frestr [n_{k+1},m_{k+1})\notin \varphi'(k)$, which indicates that $z\frestr I_k\notin \varphi(k)$ and $z\frestr I'_k\notin \varphi'(k)$. Therefore, $z\notin N_\varphi\cup N_{\varphi'}$.
\end{proof}

\begin{proof}[Proof of \autoref{alcCh}]
\ 
\noindent\eqref{alcChA}:
In view of~\autoref{Tukalcov}~\ref{Tukone}, the second inequality is easily established.  Hence, we prove the other inequality 
 \[\min\left(\{\bfrak\} \cup \set{\balc_{b, h}}{b, h\in\omega^\omega\text{
 and }\sum_{n\in\omega}\frac{h(n)}{b(n)}<\infty}\right)\leq\cov(\Nwf).\]
Assume that 
\begin{align*}\label{alcChAs}
 \kappa<\min\left(\{\bfrak\}\cup\set{\balc_{b, h}}{b,h\in\omega^\omega\text{ and }\sum_{n\in\omega}\frac{h(n)}{b(n)}<\infty}\right)\tag{$\bigstar$}   
\end{align*}
and $\set{G_\alpha}{\alpha<\kappa}\subseteq\Nwf$. We prove that $\bigcup_{\alpha<\kappa}G_\alpha\neq2^\omega$.

By~employing~\autoref{U2small}, we can find an $L\in \Ior$ and \twosmall s $\tbf^\alpha=(L,\varphi_\alpha,\varphi'_\alpha)$ for $\alpha<\kappa$ such that, for any $\alpha<\kappa$, $G_\alpha \subseteq N_{\varphi_\alpha}\cup N_{\varphi'_\alpha}$. 
By~\autoref{U2smallmain}~\ref{U2smallmainb}, we can obtain a family $\set{\varphi_{\tbf_{\alpha}}}{\alpha<\kappa}\subseteq\Swf(b,h)$ by letting $b(k):=2^{I_{k}}$ and $h(k):=\frac{2^{|I_{k}|}}{2^{k-1}}$. Then, by employing~\eqref{alcChAs}, we can find an $x\in 2^\omega$ such that, for all $\alpha<\kappa$, $\forall^\infty\, k\colon x\frestr I_k\notin\varphi_{\tbf_{\alpha}}(k)$. 
Hence, by~\autoref{U2smallmain}~\ref{U2smallmainc}, we find a family $\set{\psi^x_{\tbf_{\alpha}}}{\alpha<\kappa}\subseteq\Swf(\tilde b,\tilde h)$ by setting $\tilde b(k):=2^{|\tilde I_k|}$ and $\tilde h(k) := \frac{2^{|\tilde I_k|}}{2^{k-1}}$. Using (\ref{alcChAs}) again, we can find a $y\in2^{\omega}$ such that, for all $\alpha<\kappa$, $\forall^{\infty}\, k\in\omega\colon y\frestr\tilde I_k\notin \psi^x_{\tbf_{\alpha}}(k)$. 

Finally, 
by~\autoref{U2smallmain}~\ref{U2smallmaind}, we can find a real $z:=z_{x,y}\in 2^\omega$ such that $z\notin N_{\varphi_\alpha}\cup N_{\varphi'_\alpha}$ for all $\alpha<\kappa$. Since $G_\alpha \subseteq N_{\varphi_\alpha}\cup N_{\varphi'_\alpha}$ for all $\alpha<\kappa$, we obtain $z\notin \bigcup_{\alpha<\kappa}G_\alpha$.

\noindent\eqref{alcChB}: The first inequality follows easily by \autoref{Tukalcov}~\ref{Tukone}, so we show that
\[\non(\Nwf)\leq\sup\left(\{\dfrak\}\cup\set{\dalc_{b, h}}{b, h\in\omega^\omega\text{ and }\sum_{n\in\omega}\frac{h(n)}{b(n)}<\infty}\right).\]Pick an $\Ior$-dominating family $D\subseteq\Ior$ of partitions such that $|D|=\dfrak$ and, for all $L\in D$ and $k<\omega$, $|L_k|>k+3$. For $L\in D$ let $L_{2k}=[n^L_k,m^L_k)$, $L_{2k+1}=[m^L_k,n^L_{k+1})$ and define $I^L_k:=[n^L_k,n^L_{k+1})$ and $\tilde I^L_k:=[m^L_k,n^L_{k+1})$. Define $b^L(k):=2^{I^L_k}$, $\tilde{b}^L(k):=2^{\tilde I^L_k}$, $h^L(k)=\frac{2^{|I^L_k|}}{2^{k-1}}$ and $\tilde{h}^L(k):=\frac{2^{|\tilde I^L_k|}}{2^{k-1}}$, and choose witnesses $F^L$ and $\tilde F^L$ of $\dalc_{b^L,h^L}$ and $\dalc_{\tilde b^L,\tilde h^L}$, respectively. For $x\in F^L$ and $y\in \tilde F^L$, define $z^L_{x,y}$ as in \autoref{U2smallmain}~\ref{U2smallmaind}.

To prove the desired inequality, it is enough to show that the set \[Z:=\set{z^L_{x,y}}{x\in F^L,\ y\in\tilde F^L,\ L\in D}\] is not in $\Nwf$. Let $G\in \Nwf$, so by \autoref{U2small} there is some $L\in D$ and some \twosmall\ $\tbf=(L,\varphi,\varphi')$ such that $G\subseteq N_\varphi\cup N_{\varphi'}$. By \autoref{U2smallmain}~\ref{U2smallmainb}, $\varphi_\tbf\in\Swf(b^L,h^L)$, so there is some $x\in F^L$ such that $\forall^\infty\, k<\omega\colon x\frestr I_k\notin \varphi_\tbf(k)$. Then, by \autoref{U2smallmain}~\ref{U2smallmainc}, $\psi^x_\tbf\in\Swf(\tilde b^L,\tilde h^L)$, so there is some $y\in \tilde F^L$ such that $\forall^\infty\, k<\omega\colon y\frestr \tilde I_k\notin \psi_\tbf(k)$. Therefore, by \autoref{U2smallmain}~\ref{U2smallmaind}, $z^L_{x,y}\notin N_\varphi\cup N_{\varphi'}$, so $z^L_{x,y}\notin G$. Therefore, $Z\nsubseteq G$ for any $G\in\Nwf$, so $Z\notin \Nwf$.
\qedhere

\end{proof}

\section{Measure and weak localization}\label{sec:ELc}

We present a weakness of the localization structure and show its connection with the $\sigma$-ideal $\Ewf$ generated by the closed measure zero sets. Specifically, we use weak localization cardinals to characterize $\cov(\Ewf)$ and $\non(\Ewf)$ in \autoref{thm:chcovE}. 
The results of this section are originally due to Bartoszy\'nski and Shelah~\cite{BS1992}, and reformulated in~\cite{Car23}.


We use the following sets of slaloms to find suitable bases of $\Ewf$.

\begin{definition}\label{Setslalom}
 For an increasing function $b\in\omega^\omega$, define the following sets of slaloms:
 \begin{enumerate}[label=\normalfont(\arabic*)]

     \item $\Swf_{b}:=\set{\varphi\in\prod_{n<\omega}\Pwf(b(n))}{\forall\, n\in\omega\colon \frac{|\varphi(n)|}{2^{b(n+1)-b(n)}} \leq \frac{1}{2^n}}$.

     \item $\Swfw_{b}:=\set{\varphi\in\prod_{n<\omega}\Pwf(b(n))}{\exists^\infty\, n\in\omega\colon \frac{|\varphi(n)|}{2^{b(n+1)-b(n)}} \leq \frac{1}{2^n}}$.
 \end{enumerate}
\end{definition}

Note that $\Swf_b\subseteq\Swfw_b$ and that
\[\begin{split}
    \varphi\in\Swfw_b & \text{ implies }\liminf_{n\to\infty}\frac{|\varphi(n)|}{|b(n)|}<1,\\
    \varphi\in\Swf_b & \text{ implies }\limsup_{n\to\infty}\frac{|\varphi(n)|}{|b(n)|}<1.
\end{split}\]
Also observe that, if $b$ is a function with domain $\omega$ such that $b(i)\neq\vacio$ for all $i<\omega$, and  $h\in\omega^\omega$, then 

\begin{align*}\label{remarkE}
\liminf_{n\to\infty}\frac{h(n)}{|b(n)|}<1 & \text{\ iff\ }\forall\, m<\omega\colon \prod_{n\geq m}\frac{h(n)}{|b(n)|}=0,\tag{$\Diamond$}\\
 \limsup_{n\to\infty}\frac{h(n)}{|b(n)|}<1 & \text{\ iff\ }\forall\, A\in[\omega]^{\aleph_0}\colon \prod_{n\in A}\frac{h(n)}{|b(n)|}=0.\notag 
\end{align*}

We can constructs sets in $\Ewf$ using slaloms in the following way.

\begin{lemma}\label{toolE}
Let $\tilb\in\omega^\omega$ be increasing and $b(n):=2^{\tilb(n+1)-\tilb(n)}$ for all $n\in\omega$. If $\varphi\in\prod_{n<\omega}\pts(b(n))$ and $\liminf_{n\to\infty}\frac{|\varphi(n)|}{|b(n)|}<1$, then the set
\[H_{\tilb,\varphi}:=\set{x\in2^\omega}{\forall^\infty\, n\in\omega\colon x{\upharpoonright}[\tilb(n), \tilb(n+1))\in\varphi(n)}\]
belongs to $\Ewf$.
\end{lemma}
\begin{proof}
Notice that $H_{\tilb,\varphi}$ is a countable union of the closed null sets  
\[B_{\tilb, \varphi}^m:=\set{x\in2^\omega}{\forall\, n\geq m\colon x{\upharpoonright}[\tilb(n), \tilb(n+1))\in\varphi(n)}\text{ for }m\in\omega.\] 
Indeed,  
\begin{align*}
\Lb(B_{\tilb, \varphi}^m)=&\prod_{n\geq m}\Lb\left(\set{x\in2^\omega}{x{\upharpoonright}[\tilb(n), \tilb(n+1))\in\varphi(n)}\right)\\
=&\prod_{n\geq m}\frac{|\varphi(n)|}{2^{\tilb(n+1)-\tilb(n)}} =0,
\end{align*}
where the latter equality holds by~\eqref{remarkE}. Hence $\Lb(B_{\tilb, \varphi}^m)=0$ for all $m<\omega$, so   
\[H_{\tilb, \varphi}=\bigcup_{m<\omega} B_{\tilb, \varphi}^m\] is an $F_\sigma$ null set and thus belongs to $\Ewf$.
\end{proof}

\begin{corollary}\label{cor:toolE}
    Let $b\in\omega^\omega$ increasing. If $\varphi\in \Swfw_b$ then $H_{b,\varphi}\in\Ewf$.
\end{corollary}

Thanks to the foregoing lemma, we obtain a basis of $\Ewf$.

\begin{lemma}[{\cite[Thm. 4.3]{BS1992}}]\label{lem:combE}
Suppose that $C\in\Ewf$. Then there is some increasing $\tilb \in\omega^\omega$ and some $\varphi\in\Swf_{\tilb}$ such that $C\subseteq H_{\tilb, \varphi}$.
\end{lemma}
\begin{proof}
Let us assume wlog that $C\subseteq2^\omega$ is a null set of type $F_\sigma$. Then $C$ can be written as $\bigcup_{n\in\omega}C_n$ where $\la C_n:\, n\in\omega\ra$ is an increasing family of closed sets of measure zero. 

Note that each $C_n$ is a compact set. It is easy to see that, if $K\subseteq2^{\omega}$ is a compact null set, then 
\[\forall\, \varepsilon>0\ \forall^{\infty}\, n\ \exists\, T\subseteq2^{n}\colon K\subseteq[T] \text{ and }\frac{|T|}{2^{n}}<\varepsilon\]
where $[T]:=\bigcup_{t\in T}[t]$.
Hence, we can define and increasing $\tilb\in\omega^\omega$ such that $\tilb(0):=0$ and 
\[\tilb(n+1):=\min\set{m>\tilb(n)}{\exists\, T_n\subseteq 2^{m}\colon C_n\subseteq[T_n]\text{ and }\frac{|T_n|}{2^m} < \frac{1}{4^{\tilb(n)}}}.\]
Next, choose $T_n\subseteq2^{\tilb(n+1)}$ such that $C_n\subseteq[T_n]$  and $\frac{|T_n|}{2^{\tilb(n+1)}} < \frac{1}{4^{\tilb(n)}}$. Now define, for $n\in\omega$, \[\varphi(n):=\set{s{\upharpoonright}[\tilb(n), \tilb(n+1))}{s\in T_n}.\]
It is clear that $|\varphi(n)|\leq |T_n|$ for every $n<\omega$, hence
\begin{equation*}\label{eqE}
  \frac{|\varphi(n)|}{2^{\tilb(n+1)-\tilb(n)}}<\frac{1}{2^{\tilb(n)}}\leq\frac{1}{2^n}. 
\end{equation*}
Thus, $\varphi\in\Swf_{\tilb}$ and, by~\autoref{cor:toolE}, $H_{\tilb, \varphi}\in\Ewf$. We also have $C\subseteq H_{\tilb, \varphi}$. 
\end{proof}

The following is a technical lemma that let us connect the structure of $\omega^\omega$ with $\Ewf$. This will be essential to prove the main theorems of this section.

\begin{lemma}\label{lem:mon}
Let $b_0,b_1\in\omega^\omega$ be increasing functions, let $b_1^*\in\omega^\omega$ be as in~\autoref{Charb-d}, and let $\varphi\in\Swf_{b_0}$.
\begin{enumerate}[label=\normalfont(\arabic*)]
    \item\label{lem:monb} If $b_1\not\leq^*b_0$ then there is a $\varphi^*\in\Swfw_{b_1^*}$ such that $H_{b_0, \varphi}\subseteq H_{b_1^*, \varphi^*}$.
    
    \item\label{lem:mona} If $b_0\leq^*b_1$ then there is a $\varphi_*\in\Swf_{b_1^*}$ such that $H_{b_0, \varphi}\subseteq H_{b_1^*, \varphi_*}$.
\end{enumerate}
\end{lemma}
\begin{proof}
Let $I_n:=[b_0(n),b_0(n+1))$ and $I_n^*:=[b_1^*(n), b_1^*(n+1))$ for all $n<\omega$.

\noindent\ref{lem:monb}: According to~\autoref{Charb-d}, if $b_1\not\leq^*b_0$, then $\exists^\infty\, n\ \exists\, m\colon I_m\subseteq I_n^*$.
This allows us to define $\varphi^*(n)$ for each $n\in\omega$ as follows: \[\varphi^*(n)
   := 
    \begin{cases}
    \set{s\in 2^{I_n^*}}{I_m\subseteq I_n^*\text{ and } s{\upharpoonright} I_m\in \varphi(m)} & \text{if $m$ is minimal such that }\\
               & \text{$n\leq m$ and $I_m\subseteq I_n^*$;} \\ 
       2^{I_n^*} & \textrm{if such $m$ does not exist.}\\
    \end{cases}
\]
It remains to show that $\varphi^*\in\Swfw_{b_1^*}$, i.e.\ there are infinitely many $n$ such that, for some $m\geq n$,  $I_m\subseteq I_n^*$.
Let $0<m<\omega$ such that $b_0(m+1)<b_1(m+1)$ (there are infinitely many of those because $b_1\nleq^* b_0$), and choose $n<\omega$ such that $b_0(m)\in I_n^*$, that is, $b_1^*(n)\leq b_0(m)<b_1^*(n+1)$. We distinguish two cases:

\noindent\textbf{Case 1:} $b_0(m+1)\leq b_1^*(n+1)$. Then $b_1(n)\leq b_1^*(n)\leq b_0(m)< b_0(m+1)\leq b_1^*(n+1)$, so $I_m\subseteq I_n^*$. On the other hand $b_1(n)<b_0(m+1)<b_1(m+1)$, so $n<m+1$, i.e.\ $n\leq m$.

\noindent\textbf{Case 2:} $b_0(m+1)>b_1^*(n+1)$. We further distinguish two subcases:

\noindent\textbf{Subcase 2.1:} $b_0(m-1)\geq b_1^*(n)$. In this case, we have $I_{m-1}\subseteq I_n^*$. Then $b_1(n+1)\leq b_1^*(n+1)< b_0(m+1)<b_1(m+1)$, so $b_1(n+1)<b_1(m+1)$, which implies that $n+1<m+1$, that is, $n\leq m-1$.

\noindent\textbf{Subcase 2.2:} $b_0(m-1)< b_1^*(n)$. We will obtain a contradiction in this case.  Note that $m-1\leq b_0(m-1)<b_1^*(n)$, so $m\leq b_1^*(n)$. Then $b_1(m+1)\leq b_1(b_1^*(n)+1)=b_1^*(n+1)$.

On the other hand, $b_0(m+1)>b_1^*(n+1)$, so $b_1^*(n+1)<b_0(m+1)<b_1(m+1)\leq b_1^*(n+1)$, which gives a contradiction.

We have proved that there are infinitely many $n$ such that, for some $m\geq n$, $I_m\subseteq I_n^*$. Hence $\varphi^*\in\Swfw_{b_1^*}$. Indeed, for such $n$ and $m$, with minimal $m$,
\[\frac{|\varphi^*(n)|}{2^{|I_n^*|}}=\frac{|\varphi(m)|}{2^{|I_m|}}\leq\frac{1}{2^m}\leq\frac{1}{2^n}.\]
Therefore, $H_{b_0, \varphi}\subseteq H_{b_1^*, \varphi^*}$.

\noindent\ref{lem:mona}:  Once again, according to~\autoref{Charb-d}, note that if $b_0\leq^*b_1$ then $\forall^\infty\, n\ \exists\, m\colon I_m\subseteq I_n^*$.
For each $n\in\omega$, 
\[\varphi_*(n)
   := 
    \begin{cases}
    \set{s\in 2^{I_n^*}}{I_m\subseteq I_n^*\text{ and } s{\upharpoonright} I_m\in \varphi(m)} & \text{if $m$ is minimal such that }\\
               & \text{$n\leq m$ and $I_m\subseteq I_n^*$;} \\ 
       \emptyset & \text{if such $m$ does not exist.}\\
    \end{cases}
\]
An argument as in~\ref{lem:monb} guarantees that for all but finitely many $n$ there is some $m\geq n$ such that  $I_m\subseteq I_n^*$. This implies that $\varphi_*\in\Swf_{b_1^*}$  
and $H_{b, \varphi}\subseteq H_{b_1^*, \varphi_*}$.
\end{proof}


We present the following natural weakening of localization, which is inspired from~\cite{Br-wloc}.

\begin{definition}\label{def:Lc*}
Let $b$ be a function with domain $\omega$ such that $b(i)\neq\emptyset$ for all $i<\omega$, and let $h\in\omega^\omega$.
Define 
\[\Swfw(b,h):=\bigcup_{A\in[\omega]^{\aleph_0}}\Swf(b\frestr A,h\frestr A),\]
i.e., the set of all partial slaloms ``in $\Swf(b,h)$" with infinite domain.
Let $\wLc(b,h):=\la\prod b,\Swfw(b,h),\in^*\ra$ be the relational system
where $x\in^* \varphi$ iff $\forall^\infty\, i\in\dom\varphi\colon x(i)\in \varphi(i)$ (extending the notation in \autoref{def:LcaLc}). Denote  $\bloc_{b,h}:=\bfrak(\wLc(b,h))$ and $\dloc_{b,h}:=\dfrak(\wLc(b,h))$. 
\end{definition}




We can easily prove:

\begin{fct}\label{fct:Lc*}
    For any $A\subseteq\omega$ infinite, $\wLc(b,h)\leqT \Lc(b\frestr A, h\frestr A)$. In particular, $\dloc_{b,h}\leq \dlc_{b\frestr A, h\frestr A}$ and $\blc_{b\frestr A, h\frestr A}\leq \bloc_{b,h}$.
\end{fct}

We obtain the following simple connection between $\Ewf$ and weak localization.

\begin{lemma}\label{uppbE}
If $b\in\omega^\omega$ and $\limsup_{n\to\infty}\frac{h(n)}{b(n)}<1$, then $\Cbf_\Ewf\leqT\wLc(b,h)$. In particular, $\cov(\Ewf)\leq\dloc_{b,h}$ and $\bloc_{b,h}\leq\non(\Ewf)$.
\end{lemma}
\begin{proof}
Since $\Cwf_{\Ewf(2^\omega)}\eqT\Cwf_{\Ewf(\prod b)}$ by~\autoref{TukeyPolish}~\ref{TukeyPolishc}, we just need to prove that $\Cbf_{\Ewf(\prod b)}\leqT\wLc(b,h)$. We find a function  $\Psi_+\colon \Swfw(b,h)\to\Ewf(\prod b)$ such that for any $x\in\prod b$ and $\varphi\in\Swfw(b, h)$, if $x\in^*\varphi$, then $x\in\Psi_+(\varphi)$.

For $\varphi\in \Swfw(b,h)$ put 
\[\Psi_+(\varphi):=\set{x\in\prod b}{\forall^\infty\, n\in \dom\varphi\colon x(n)\in \varphi(n)}.\] 
We only have to check that $\Psi_+(\varphi)\in\Ewf(\prod b)$. Indeed, setting $A_m:=(\dom\varphi)\menos m$ and
\[B_{\varphi,m}:=\set{x\in\prod b}{\forall\, n\in A_m\colon x(n)\in\varphi(n)},\] 
we have
\begin{align*}
\Lb(B_{\varphi,m})&=\prod_{n\in A_m}\Lb\left(\set{x\in\prod b}{x(n)\in\varphi(n)}\right)\\
&=\prod_{n\in A_m}\frac{|\varphi(n)|}{b(n)}\\
&\leq\prod_{n\in A_m}\frac{h(n)}{b(n)}=0.
\end{align*}
(The last equality follows by~\eqref{remarkE}). Hence, $\Lb(B_{\varphi,m})=0$ for all $m\in\omega.$ Thus $\Psi_+(\varphi)=\bigcup_{m<\omega} B_{\varphi,m}$ is an $F_\sigma$ null set and belongs to $\Ewf(\prod b)$.
%
\end{proof}

We are finally ready to prove the main result of this section.

\begin{theorem}[{\cite{BS1992},~\cite{Car23}}]\label{thm:chcovE}
\begin{multline*}\tag{A}
 \min\left(\{\dfrak\}\cup\set{\dloc_{b, h}}{b, h\in\omega^\omega\text{ and }\limsup_{n\to\infty}\frac{h(n)}{b(n)}<1}\right)\leq\cov(\Ewf)\\
 \leq\min\set{\dloc_{b, h}}{b, h\in\omega^\omega\text{ and }\limsup_{n\to\infty}\frac{h(n)}{b(n)}<1}.
 \label{ELc:A}
\end{multline*}
In particular, if  $\cov(\Ewf)<\dfrak$ then \[\cov(\Ewf)=\min\set{\dloc_{b, h}}{b, h\in\omega^\omega\text{ and }\limsup_{n\to\infty}\frac{h(n)}{b(n)}<1}.\]
\begin{multline*}\tag{B}
\sup\set{\bloc_{b, h}}{b, h\in\omega^\omega\text{ and }\limsup_{n\to\infty}\frac{h(n)}{b(n)}<1}\leq\non(\Ewf)\leq\\
  \sup\left(\{\bfrak\}\cup\set{\bloc_{b, h}}{b, h\in\omega^\omega\text{ and }\limsup_{n\to\infty}\frac{h(n)}{b(n)}<1}\right).
  \label{ELc:B}
\end{multline*}
In particular, if $\non(\Ewf)>\bfrak$ then \[\non(\Ewf)=\sup\set{\bloc_{b, h}}{b, h\in\omega^\omega\text{ and }\limsup_{n\to\infty}\frac{h(n)}{b(n)}<1}.\]
\end{theorem}
\begin{proof}
\ 
\eqref{ELc:A}
The second inequality is clear from~\autoref{uppbE}. So we prove the other inequality \[\cov(\Ewf)\geq\min\left(\{\dfrak\}\cup\set{\dloc_{b, h}}{b, h\in\omega^\omega\text{ and }\limsup_{n\to\infty}\frac{h(n)}{b(n)}<1}\right).\] 
To see this, assume $\kappa<\min\left(\{\dfrak\}\cup\set{\dloc_{b, h}}{b, h\in\omega^\omega\text{ and }\limsup_{n\to\infty}\frac{h(n)}{b(n)}<1}\right)$ and $\Jwf\subseteq\Ewf$ of size $\kappa$ . We will prove that $\bigcup\Jwf\neq2^\omega$.

For each $A\in\Jwf$, by \autoref{lem:combE}, there are $b_A\in\omega^\omega$ increasing and $\varphi_A\in\Swf_{b_A}$ such that $A\subseteq H_{b_A, \varphi_A}$. Since $|\Jwf|<\dfrak$, we can find an increasing function $b$ such that $b\not\leq^* b_A$ for $A\in\Jwf$. Then, by \autoref{lem:mon}, we obtain a family $\{\varphi_{A}^*:\, A\in\Jwf\}\subseteq\Swfw_{b^*}$ such that $A\subseteq H_{b^*, \varphi_A^*}$.
Observe that this family corresponds to a family of partial slaloms in $\Swfw(b', h)$ of size ${\leq}\kappa$ with $b'(n):=2^{b^*(n+1)-b^*(n)}$ and $h(n):=\frac{b'(n)}{2^n}$ for all $n$.
Therefore, we can find an $x\in\prod b'$ such that, for all $A\in\Jwf$, $\exists^\infty\, n<\omega\colon x(n)\notin\varphi^*_A(n)$. Let $y\in2^\omega$ be such that $x(n)=y{\upharpoonright}[b'(n), b'(n+1))$ for all $n$. It is clear that $y\not\in H_{b^*, \varphi^*_{A}}$ for all $A\in\Jwf$, so $y\notin \bigcup \Jwf$.

\eqref{ELc:B}: The first inequality is clear from \autoref{uppbE}. To see the other inequality, fix an unbounded family $B\subseteq\omega^\omega$ of increasing functions of size $\bfrak$. For each $b\in B$, define $b'(n):=2^{b^*(n+1)-b^*(n)}$ and $h_b(n):=\frac{b'(n)}{2^n}$ for all $n<\omega$, and pick a witness $F_b\subseteq \prod b'$ of $\bloc_{b',h_b}$. For each $x\in F_b$, define $y_x\in 2^\omega$ such that $x(n)=y_x\frestr[b'(n),b'(n+1))$. Let
\[H:=\set{y_x}{x\in F_b,\ b\in B}.\]
It is enough to show that $H\notin\Ewf$. Indeed, if $A\in\Ewf$, by \autoref{lem:combE}, there are $b_A\in\omega^\omega$ increasing and $\varphi\in\Swf_{b_A}$ such that $A\subseteq H_{b_A, \varphi}$. Then, there is some $b\in B$ such that $b\not\leq^* b_A$ so, by \autoref{lem:mon}, there is some $\varphi^*\in \Swfw_{b^*}$ such that $A\subseteq H_{b^*,\varphi^*}$. Since $\varphi^*$ corresponds to a member of $\Swfw(b',h_b)$, there is some $x\in F_b$ such that $x\notin^* \varphi^*$, so $y_x\notin H$. This shows that $A\nsubseteq H$ for all $A\in \Ewf$, so $H\notin\Ewf$.
\end{proof}

We review (without proof) the following results about the cardinal characteristics associated with $\Ewf$.

\begin{theorem}[{\cite{BS1992}}, see also~{\cite[Sec.~2.6]{BJ}}]\label{thm:E}
\ 
  \begin{enumerate}[label =\normalfont (\alph*)]
      \item $\min\{\bfrak,\non(\Nwf)\} \leq\non(\Ewf)\leq\min\{\non(\Mwf),\non(\Nwf)\}$.
      \item $\max\{\cov(\Mwf),\cov(\Nwf)\}\leq\cov(\Ewf)  \leq\max\{\dfrak,\cov(\Nwf)\}$.
      \item $\add(\Ewf)=\add(\Mwf)$ and $\cof(\Ewf)=\cof(\Mwf)$.
  \end{enumerate}
\end{theorem}

One could ask whether we can change the weak localization cardinals by localization cardinals in \autoref{thm:chcovE}. By \autoref{fct:Lc*} we have
\begin{align*}
    \cov(\Ewf) & \leq \min\set{\dlc_{b, h}}{b, h\in\omega^\omega\text{ and }\limsup_{n\to\infty}\frac{h(n)}{b(n)}<1} \text{ and }\\
    \non(\Ewf) & \geq \sup\set{\blc_{b, h}}{b, h\in\omega^\omega\text{ and }\limsup_{n\to\infty}\frac{h(n)}{b(n)}<1},
\end{align*}
but the other inequalities, with localization instead of weak localization, do not hold. To see this, note that the standard $\sigma$-centered poset $\mathbb{E}$ that adds an eventually different real on $\omega^\omega$ increases $\non(\Ewf)$, while it does not increase $\bfrak$ nor $\blc_{b,h}$ for $b,h\in\omega^\omega$ with $0<h<b$. Therefore, if $\aleph_1<\kappa<\lambda=\lambda^{\aleph_0}$ are cardinals with $\kappa$ regular, then the FS iteration of $\mathbb{E}$ of length $\lambda\kappa$ (ordinal product) forces $\bfrak=\blc_{b,h}=\aleph_1$ and $\dfrak=\dlc_{b,h}=\lambda$ for all $b,h\in\omega^\omega$ with $0<h<b$, and $\non(\Ewf)=\cov(\Ewf)=\kappa$.

\section{Null-additive and localization}\label{sec:other}


We establish the characterization of the uniformity of the null-additive ideal in terms of the localization cardinals. The proof is based on Shelah's characterization of null-additive sets, which is simpler than the original Pawlikowski's~\cite{paw85} proof.

\begin{definition}
A set $X\subseteq2^\omega$ is \emph{null-additive} if $X+A\in\Nwf$ for all $A\in \Nwf$. (Here the sum on $2^\omega$ is the coordinate-wise sum modulo $2$.) Denote by $\NAwf$ the $\sigma$-ideal of the null-additive sets in $2^\omega$.
\end{definition}

Let $I=\seq{I_n}{n\in\omega}$ be a sequence of finite subsets of $\omega$. For $\varphi\in\prod_{n<\omega}\pts(2^{I_n})$, we define
\[
H_\varphi^*:=\set{x\in\cantor}{\forall^{\infty}\, n\in \omega\colon x{\upharpoonright}I_n\in \varphi(n)}.\\
\]

We employ the following characterization of $\NAwf$ due to Shelah.

\begin{theorem}[{\cite{shmn}}]\label{Shmn}
A set $X\subseteq 2^\omega$ is in $\NAwf$ iff for all $I=\seq{I_n}{n\in\omega}\in\Ior$ there is some $\varphi\in\prod_{n\in \omega}\pts(2^{I_n})$ such that $X\subseteq H_\varphi$ and  $\forall\, n\in \omega\colon |\varphi(n)|\leq n$.    
\end{theorem}




The uniformity of $\NAwf$ is characterized as follows.

\begin{theorem}[{\cite[Lem.~2.2]{paw85}}]
 $\non(\NAwf)=\minLc$.   
\end{theorem}
\begin{proof}
For the proof, it is enough to take $b$ from some dominating family, so 
we may assume that there is an $I\in\Ior$ such that $b(n)=2^{|I_n|}$  for all $n<\omega$. We first prove that $\Lc(b,\id_\omega)\leqT\Cbf_{\NAwf}$. To achieve this, it suffices to  find functions $\Psi_-\colon \prod b\to2^\omega$ and $\Psi_+\colon \NAwf\to\Swf(b,\id_\omega)$ such that, for any $x\in\prod b$ and  $X\in \NAwf$, if $\Psi_-(x)\in X$, then $x\in^*\Psi_+(X)$. 

Let $X\in\NAwf$. Then, by applying~\autoref{Shmn}, we can find $\varphi_X\in\prod_{n\in \omega}\pts(2^{I_n})$ such that $X\subseteq H_{\varphi_X}^*$ and  $\forall n\in \omega:|\varphi_X(n)| \leq n$. It is clear that $\varphi_X\in\Swf(b,\id_\omega)$. Put $\Psi_+(X):=\varphi_X$

On the other hand, 
for $x\in\prod b = \prod_{i<\omega}2^{I_i}$, define $\Psi_-(x)\in 2^\omega$ a the concatenation of the binary sequences $\seq{x(n)}{n<\omega}$. It is not hard to see that $(\Psi_-,\Psi_+)$ is the required Tukey connection. 

Now we prove the converse inequality, that is, $\min\set{\blc_{b,\id_\omega}}{b\in\omega^\omega}\leq\non(\NAwf)$. Suppose that $X\subseteq2^{\omega}$ is a set of size ${<}\minLc$, and we show that $X\in\NAwf$. Assume $I\in\Ior$ and let $b(n):=2^{|I_n|}$ for all $n\in\omega$. For $x\in X$, let $\bar x=\seq{x{\upharpoonright}I_n}{n<\omega}$, so $\bar x\in\prod b$. Since $\set{\bar x}{x\in X}$ has size ${<}\minLc$, we can find some $\varphi\in\Swf(b,\id_\omega)$ such that $\bar x\in^* \varphi$ for all $x\in X$. Hence, it is clear that $X\subseteq H_\varphi^*$ and $\forall n\in \omega:|\varphi(n)|\leq n$, so by~employing~\autoref{Shmn} we obtain that $X\in\NAwf$. 
\end{proof}

The second part of the previous proof actually states that $\Cbf_{\NAwf} \leqT \prod_{b\in D}\Lc(b,\id_\omega)$ (direct product of relational systems) where $D$ is some dominating family in $\omega^\omega$ such that, for each $b\in D$, there is some $I\in\Ior$ such that $b(n)=2^{|I_n|}$  for all $n<\omega$.

\subsection*{Acknowledgments}

This paper was developed for the conference proceedings of the Kyoto University RIMS Set Theory Workshop ``New Developments in Forcing and Cardinal Arithmetic", organized by Professor Hiroshi Sakai from Kobe University. The authors are very thankful to him to organize and letting them participate in such a wonderful workshop, and for the financial support the first author received to attend physically at Kyoto University RIMS.

{\small
\bibliography{left}
\bibliographystyle{alpha}
}

\end{document}